\documentclass{amsart}

\usepackage[utf8]{inputenc} 
\usepackage[english]{babel}
\usepackage[T1]{fontenc}

\usepackage{url}

\usepackage{amsmath}
\usepackage{amsfonts}
\usepackage{amssymb}
\usepackage{pst-node}
\usepackage{tikz-cd}
\usepackage{tikz}
\usetikzlibrary{trees}
\usetikzlibrary{arrows}
\usepackage{amsthm}
\usepackage{mathrsfs}
\usepackage{bm}
\usepackage{mathtools}
\usepackage[framemethod=TikZ]{mdframed}
\usepackage{ytableau}

\usepackage{units}

\usepackage{listings}
\usepackage[margin=60pt]{geometry}

\usepackage[babel]{csquotes}
\usepackage{hyperref}
\usepackage[
backend=biber,
style=numeric,
sorting=nty,
maxnames=10
]{biblatex}
\input xy
\xyoption{all}

\usepackage{graphicx}

\theoremstyle{theorem}
\newtheorem{thm}{Theorem}[section]
\newtheorem{conj}[thm]{Conjecture}
\newtheorem{prp}[thm]{Proposition}
\newtheorem{crl}[thm]{Corollary}
\newtheorem{lemma}[thm]{Lemma}

\theoremstyle{definition}
\newtheorem{df}[thm]{Definition}
\newtheorem{ex}[thm]{Example}

\theoremstyle{remark}

\newtheorem{rmk}[thm]{Remark}

\newcommand{\Hilb}{\operatorname{Hilb}}

\newcommand{\KHilb}[1]{\operatorname{Hilb}^{#1}_{0,k}(C)}
\newcommand{\PHilb}[1]{\operatorname{Hilb}^{#1}_{0,1}(C)}

\newcommand{\UHilb}{\operatorname{UHilb}}
\newcommand{\Li}{\mathbb L}
\newcommand{\C}{\mathbb K}
\newcommand{\Lo}{\mathcal{L}}
\newcommand{\Cont}{\mathcal{X}}
\newcommand{\ring}{\widehat{\mathcal{O}}_{C,0}}

\usepackage{xcolor}

\begin{document}

\title[Motivic classes of fixed-generators Hilbert schemes]{Motivic classes of fixed-generators Hilbert schemes of unibranch curve singularities and Igusa zeta functions}

\author{Ilaria Rossinelli}
\address{Department of Mathematics, EPFL Lausanne}
\email{ilaria.rossinelli@epfl.ch}
\date{}
\maketitle

\begin{abstract}
This paper delves into the study of Hilbert schemes of unibranch curves whose points have a fixed number of minimal generators. Building on the work of Oblomkov, Rasmussen, and Shende, we provide a formula for their motivic classes and investigate the relationship with the principal Hilbert schemes of the same given unibranch curve. 

In addition, the paper specializes this study to the case of $(p,q)$-curves, where we obtain more structured results for the motivic classes of fixed-generators Hilbert schemes: their positivity and topological invariance, and an explicit relationship to one-generator schemes or equivalently principal ideals in $\ring$.

Finally, we focus on a special open in the one-generator locus, whose motivic class is naturally related to the motivic measure on the arc scheme $\mathbb A^2_\infty$ of the plane introduced by Denef and Loeser as well as to the Igusa zeta function. We also provide an explicit formulation of these motivic classes in terms of an embedded resolution of the singularity, proving their polynomiality and making them an interesting topological invariant of the given curve.
\end{abstract}

\tableofcontents

\section{Introduction} \label{intro}

Given a curve singularity, we are naturally interested in studying its geometric and topological properties. 
Algebraic links, originally presented in \cite{milnor} as the topological data of the given curve, are in perfect correspondence with the geometry of the singularity as proved in \cite[Chapter III.8]{Brieskorn}. For a given germ $(C, 0)$ of an isolated reduced complex plane curve singularity, its link is defined as the nonsingular one-dimensional real submanifold of the sphere with radius $\epsilon \ll 1$ 
\[ L = S^3_{\epsilon , 0} \cap C .\] 

Despite being easy to define, links turn out to be very hard to be handled directly, and the associated (co)homological invariants such as the HOMFLY polynomial and the HOMFLY homology groups still lack an interpretation in terms of the given singularity predicted by the equivalence in \cite{Brieskorn}. 

Starting in \cite{cdgz} and consequently in the works \cite{os}, \cite{ors}, and many more, Hilbert schemes have been considered as a possible geometric setting to answer these interpretation questions. Hilbert schemes whose points have a fixed number of minimal generators are defined as 
$$\KHilb{n}= \{ I \in \Hilb^n_0(C) \, | \, \dim_\C I/m_0I =k  \} $$ 
where $m_0 = (x,y) \subseteq \ring$ denotes the maximal ideal of the local ring $\ring$ of the singularity, so that by Nakayama's lemma $\dim_\C I/m_0I$ is the minimal number of generators of $I$ (see Definition \ref{def:number-gens}). These Hilbert schemes refine $\Hilb^n_0(C)$
and are, in particular, the central object of interest in the conjecture proposed by Oblomkov, Rasmussen and Shende in \cite{ors}. This conjecture originates as a homological enhancement of the version presented in \cite{os} and proved in \cite{mau}, and relates the virtual Poincaré polynomials $\omega(\KHilb{l})$ with the superpolynomial $\mathcal{P}_L$ associated to the triply-graded homology of the link of the curve singularity.

In this paper, we focus on the former, studying the \textit{motivic class} of $\KHilb{n}$, that can be later specialized to its weight polynomial.

We begin with Section \ref{sec:motivic-classes} and our first result, which is the existence and uniqueness of some special representatives for the points of $\KHilb{n}$ using their $t$-adic valuations in $\ring \subseteq \C[[t]]$ and the semigroup ideals of $\Gamma_C$. 

\begin{thm}
Let $(C,0)$ be a unibranch curve.
Let $I \in \KHilb{n}$. Then, there exist a minimal set of generators $g_1 , \dots , g_k$ such that the values $val_t(g_i) = n_i$ form an independent $k$-uple in the sense of Definition \ref{def:independent-int-numbers}, and this is the only possible independent $k$-uple of valuations arising from minimal sets of generators of $I$. 
\end{thm}


The computation is carried out also in Section \ref{sec:motivic-classes}, where we explicitly present the motivic classes of the partition in terms of certain subvarieties of determinantal varieties arising from the syzygies encoding the $\ring$-combinations corresponding to valuations in $\Gamma$-semigroup ideals. We obtain the following result. 

\begin{thm} \label{main-formula unibranch}
Let $(C,0)$ be a unibranch curve. We have $$[\KHilb{n}] 
= \sum_{\Delta} 
\Li^{\sum_i |\Gamma_{> \nu_i} \setminus \Delta |}   
\sum_{\underline{n}} 
\Li^{- \sum_i |\{n_i+H \} \cap \Delta|}
\sum_{a = 1 \dots m_\Delta} 
[S_a]\Li^{ - a } $$
for every $n,k \geq 1$. Here $\Delta$ runs over the semigroup ideals of the semigroup $\Gamma = \Gamma_C$ of $(C,0)$ (Definition \ref{def-semigroup}) with $|\Gamma \setminus \Delta| = n$, and $\nu_1 , \dots , \nu_l$ denote the generators of $\Delta$ as a semigroup ideal (Definition \ref{ors-phi-gens}), over which the index $i$ of the first exponent runs. The $k$-uple $\underline{n} = (n_1, \dots , n_k)$ runs over the independent $k$-uples of Definition \ref{def:independent-int-numbers} whose associated semigroup ideal is contained in $\Delta$, and the index $i$ of the second exponent runs over $1, \dots , k$. Moreover, $H = \mathbb N \setminus \Gamma$ is the set of holes of $\Gamma$, $\Gamma_{>\nu} = \{ j \in \Gamma \, | \, j > \nu \}$, $m_\Delta$ is the number of syzygies of $\Delta$ of Definition \ref{ors-phi-gens}, the affine space $Syz$ of the syzygy variables is introduced in Definition \ref{notation}, the matrix $M_\Delta(p)$ is described in Corollary \ref{crl:stratum-class} and Remark \ref{rmk:linearity}, and
$$S_a = \{ p \in Syz \, | \, rk(M_\Delta (p))=a \} .$$
\end{thm}

In Section \ref{sec:pq-curves}, we specialize the previous discussion to the simpler case of $(p,q)$-curves, that is, of curves whose local ring is a monomial ring $\ring = \C[[t^p,t^q]]$ for two coprime integers $p<q$ (Definition \ref{def:pq-curve}), as the semigroup ideals of $\Gamma_C$ are easier to list, without involving the additional structure coming from non-trivial combinations in $\ring$ as before. Consequently, we obtain the following expression for $[\KHilb{n}]$.

\begin{thm} \label{main-formula-pq}
Let $(C,0)$ be a $(p,q)$-curve. We have 
$$[\KHilb{n}]  = \sum_{\underline{n} \in M_n} \left( \prod_i [\PHilb{{n_i}}] \right) \Li^{- \sum_i |\{ n_i + H \} \cap \Gamma_{\underline{n}}|} = \sum_{\underline{n} \in M_n} \Li^{ \sum_i |\{ n_i + H \} \cap \Gamma_C|  - \sum_i |\{ n_i + H \} \cap \Gamma_{\underline{n}}|}$$ for every $n,k \geq 1$ where
$$\Gamma_{\underline{n}}  = \coprod_i \{ n_i + \Gamma_C \}  .$$
\end{thm}

The positivity (of the coefficients of $[\KHilb{n}]$ seen as a polynomial in $\Li$, see \cite[Proposition 6]{ors}), the topological invariance, and the relationship to the motivic class of the one-generator locus immediately follow from the theorem.

\begin{crl}
Let $(C,0)$ be a $(p,q)$-curve. Then, $[\KHilb{n}]$ is a polynomial in $\Li$ with positive coefficients for every $n,k \geq 1$. 
\end{crl}


Moreover, building on the work presented in \cite{ila-curv}, we aim to further develop the theory of motivic integration \textit{in relationship to} Hilbert schemes, 
and specifically express the classic motivic Igusa zeta function
$$Z_C(s) = \sum_{n=1}^\infty \mu(\Cont^n_C) \Li^{-ns}$$
defined in \cite{DL} in terms of motivic classes of principal Hilbert schemes.
Precisely, we want to understand the motivic measure $\mu(\Cont^n_C)$, defined in \cite{mea} and \cite{DL}, of the contact locus $\Cont^n_C \subseteq \Lo$ of a given unibranch curve $C$ in connection with $\PHilb{n}$.
This is accomplished in Section \ref{sec:contact-loci-principal}, where we construct a morphism 
$$\Pi: \pi_{n+1}(\Cont^n_C) \to \operatorname{Im}(\Pi) \subseteq \PHilb{n}$$
realizing pairs of (truncated) arcs as (intersections of the given curve $C$ with) unibranch curves. 
The morphism $\Pi$ is not surjective in general. However, the ``unibranch'' locus $$\UHilb^n_{0,1}(C) := \operatorname{Im}(\Pi) \subseteq \PHilb{n}$$
is enough to retrieve the coefficients of $Z_C(s)$. 

To be more precise, the relationship between $\UHilb^n_{0,1}(C)$ and $Z_C(s)$ is obtained thanks to the explicit description of the fibers of $\Pi$, based on \cite{luengo} and the so-called space of branches presented therein. We will see how the morphism $\Pi$ descends to a geometric morphism in the space of branches, directly leading to the description of $[\UHilb^n_{0,1}(C)]$ in terms of the (truncation of the) contact locus $[\pi_{n+1}(\Cont^{n})]$ of the given curve.
\begin{thm} \label{main-contactloci-principal}
Let $(C,0)$ be a unibranch curve. We have
    $$[\operatorname{Im}(\Pi_b^k)] 
     = \frac{\Li^{k}}{\Li^n(\Li-1)}  \frac{ [\pi_{n+1}(\Cont^{n,k})]}{[F^{n,k}]}  
   $$
for every $n,k \geq 1$ where
$$ \Cont^{n,k}_C = \Cont^n_C \cap \pi_k^{-1}(0) \setminus \Cont^n_C \cap \pi_{k+1}^{-1}(0) $$
is a subset of the contact locus $\Cont^n_C$ defined in Definition \ref{def:contact-locus}, $F^{n,k}$ is defined in Proposition \ref{prp-def-Fnk}, $\Pi_b$ denotes the $S_0^{n+1}$-quotient of the map $\Pi$,
and $[\operatorname{Im}(\Pi_b^k)]$ is related to $[\UHilb_{0,1}^n(C)]$ in Proposition \ref{legame}.
\end{thm}

Combining each (truncation of) contact locus in their generating series, we obtain the classic Igusa zeta function $Z_C(s)$, and the coefficients of this series are now related to Hilbert schemes.

\begin{crl}
We have
$$Z_C(s)= (\Li-1) \Li^{-2}\sum_{n \geq 1}\Li^{-n(s+1)}\sum_k \Li^{-k}[\operatorname{Im}(\Pi_b^k)] [F^{n,k}]  . $$
\end{crl}
In the same section, we also observe that the components $\Cont^{n,k}$ in the $n$-th contact locus and their motivic measure are completely determined by an embedded resolution of $(C,0)$. This leads to a relationship of $[\UHilb^n_{0,1}(C)]$ with an embedded resolution, hence becoming an interesting topological polynomial invariant.
\begin{crl}
Let $(C,0)$ be a unibranch curve. We have
$$[\operatorname{Im}(\Pi_b^k)] = \frac{1}{(\Li-1)\Li^{n}} \frac{\Li^k}{[F^{n,k}]}
\sum_{\substack{J \subseteq I, J \neq \emptyset, j \in J \\ k_j \geq 1 \, | \, k =\sum m_j k_j}}
(\Li-1)^{|J|-1}[E_J^o]
\left( \sum_{j \in J, \, \sum k_jN_j = n}
\Li^{-\sum_{j \in J} k_j \nu_j} \right)$$
for every $n, k \geq 1$. Here $I$ indexes the irreducible components $E_j$ of the total transform $\rho^{-1}(C)$ of an embedded resolution $\rho$ as in Definition \ref{res}, $E_J^o$ is the open stratum associated to $J \subseteq I$ of Definition \ref{def:strata}, $N_j$ and $\nu_j$ are respectively the multiplicity and the discrepancy of $E_j$ of Definition \ref{mult}, the class $[F^{n,k}]$ of the fiber of $\Pi^k_b$ is described in Proposition \ref{fibers-pi} and Remark \ref{rmk:explicit-Fnk}, and $m_j$ is defined in \cite[Definition 3.21]{ila-curv}. 
\end{crl}
\begin{crl}
Let $(C,0)$ be a unibranch curve. Then, $[\UHilb^n_{0,1}(C)]$ is a polynomial in $\Li$ for every $n \geq 1$. This polynomial is also a topological invariant of $(C,0)$ since the resolution is. 
\end{crl}

Finally, we also prove the existence of a threshold from which $\UHilb^n_{0,1}(C)$ is non-empty, making it interesting to study not only in relationship to the Igusa zeta function.

\begin{thm} \label{main-unibranch-reps}
    Let $(C,0)$ be a unibranch curve and let $\rho: (Y, E) \rightarrow (\mathbb A^2, C)$ be an embedded resolution of the pair $(\mathbb A^2, C)$ as in Definition \ref{res}.
    Let $N_i$ denote the multiplicity of $E_i$ as in Definition \ref{mult} and let $N = \operatorname{max}_i N_i$.  
    Then, $$\UHilb^n_{0,1}(C) \neq \emptyset$$ for every $n \geq N$.   
\end{thm}


\vspace{4.5mm}
This paper is organized as follows. We recall the main notions and results of the theory of motivic integration, unibranch curve singularities, and Hilbert schemes in Section \ref{sec:background}.
Then, the first part of Section \ref{sec:motivic-classes} is dedicated to constructing a partition of the Hilbert scheme of a given unibranch curve with a fixed number of generators, determined by $t$-adic valuations of the generators in $\ring$ and semigroup ideals.
Later, we study the relationship of motivic classes of fixed-generators Hilbert schemes to certain determinantal subvarieties to prove Theorem \ref{main-formula unibranch}. 
In Section \ref{sec:pq-curves} we then specialize the results of Section \ref{sec:motivic-classes} to the case of $(p,q)$-curves, and factoring in some specific properties of these simpler curves we obtain Theorem \ref{main-formula-pq}.
Finally, the first part of Section \ref{sec:contact-loci-principal} is devoted to studying the existence of irreducible representatives for points of principal Hilbert schemes of a given unibranch curve, leading to Theorem \ref{main-unibranch-reps}. We conclude by relating the motivic classes of the unibranch locus of the one-generator locus to the motivic measure of contact loci of the given unibranch curve in Theorem \ref{main-contactloci-principal}.
We conclude this paper with a wide range of examples, presented in Section \ref{sec:ex}, where we set out our various results for different $(p,q)$-curves, until we are able to reconstruct their HOMFLY-PT polynomial using the new techniques presented in this paper.


\vspace{4.5mm}
\indent
\textbf{Acknowledgements.} I would like to deeply thank Dimitri Wyss for his advice and supervision during this paper's discussion and writing stages. Moreover, we would like to thank Alexei Oblomkov for his valuable suggestions and insight, and Gergely Bérczi for the helpful discussions.
This work was supported by the Swiss National Science Foundation [No. 196960].

\section{Background} \label{sec:background}

We denote by $\C$ an algebraically closed field of characteristic zero. We denote by $\operatorname{Sch}_\C$ the category of Noetherian schemes over $\C$ and by $\operatorname{Var}_\C$ the category of varieties i.e. separated integral schemes of finite type over the same field $\C$. Throughout this paper, when referring in general to a scheme or a variety we imply objects of these categories and over $\C$.
Moreover, by curve or simply curve, we will always mean an hypersurface in $\mathbb A^2_\C = \mathbb A^2$, that we will assume to be reduced and irreducible throughout this paper. 

\subsection{Motivic measure and contact loci}
In this section, we recall some basic facts about the construction of the motivic measure and introduce the notion of contact loci. We refer to \cite{greenbook} or \cite{DL} as references on the subject.

We start with the fundamental notions of $k$-jet scheme and arc scheme, and their truncation maps, which will have a central role in the definition of motivic integrals. 

\begin{df} \label{def-the jet scheme}
Let $X, Y$ be schemes. For any $k > 0$ we denote by $X_k(Y)$ the set 
$$X_k(Y) = \operatorname{Hom}_{\operatorname{Sch}_\C}(Y \times_\C \operatorname{Spec}(\mathbb \C[t] /(t^k)) , X)$$
and the $k$-jet scheme is defined as the scheme representing the corresponding functor $ X_k : \operatorname{Sch}_\C^{op} \to \operatorname{Set} $ defined on objects by $ Y \mapsto X_k(Y)$.
\end{df}

This definition relies on the representability result of the functor $X_k$, which one can find in \cite[Proposition 2.1.3]{greenbook}. The scheme representing the functor $X_k$ will also be denoted by $X_k$ and it will be called the $k$-jet scheme, and we call $k$-jet a point of $X_k$. The $\C$-points of $X_k$ are given by $X_k(\C) = \operatorname{Hom}_{\operatorname{Sch}_\C}(\operatorname{Spec}(\C[t] /(t^k)) , X)$ and correspond to $\C[t] / (t^k)$-points of $X$.

\begin{df}
Let $l \geq r >0$ be two integers and consider the map $\pi^l_r : \C[t] / (t^l) \to \C[t] / (t^r)$ of reduction modulo $t^r$. We denote also by $\pi^l_r: X_l \to X_r$ the induced morphism of schemes.
\end{df}

Following \cite[Section 3.3.1, Corollary 3.3.7]{greenbook} the truncation morphisms $(\rho^{k+1}_k)_{k \geq 0}$ form a projective system which admits a limit in the category of functors over $\operatorname{Sch}_\C$. The limit functor is again representable and the scheme representing it can be taken as the definition of the arc scheme. In particular, we summarize the key points of the arc scheme construction in the next definition.

\begin{df} \label{def-the arc scheme}
The arc scheme is defined as the scheme representing the limit functor $\varprojlim X_k = X_\infty : \operatorname{Sch}_\C^{op} \to \operatorname{Set}$. In particular, the functor on objects is defined as $$X_\infty(Y) = \operatorname{Hom}_{\operatorname{Sch}_\C}(Y \times_\C \operatorname{Spec}(\mathbb \C[[t]]) , X).$$
\end{df}

The scheme representing the functor $X_\infty$ will also be denoted by $X_\infty$, and we call arc a point of $X_\infty$. 
We observe that in this limit case as well the $\C$-points of $X_\infty$ are given by $X_\infty(\C) = \operatorname{Hom}_{\operatorname{Sch}_\C}(\operatorname{Spec}(\C[[t]] ) , X)$ and correspond to $\C[[t]]$-points of $X$ as explained in \cite[Remark 3.3.9]{greenbook}.

\begin{df} \label{tr-arcs}
Let $ r >0$ be an integer and consider the map $\pi_r : \C[[t]] \to \C[t] / (t^r)$ of reduction modulo $t^r$. We denote also by $\pi_r: X_\infty \to X_r$ the induced morphism of schemes.
\end{df}

\begin{ex} \label{the jet scheme-aff}
We present the example of the affine space and more in general affine varieties since this ultimately is our local model of varieties.
Let $X = \mathbb A^n$ with coordinates $x_1, \dots , x_n$. We have that a morphism $ \phi : \C[x_1, \dots , x_n] \rightarrow \C[t] / (t^k)$ is uniquely determined by the images $\phi(x_i)=a^i \in \C[t] / (t^k)$ of the coordinates, each given by a $k$-tuple of coefficients $(a^i_0, \dots, a^i_{k-1}) \in \C^k$, that is, by the truncated arc $t \mapsto (a^1(t), \dots , a^n(t))$.
Therefore, we have that 
$$ (\mathbb A^n)_k = (\mathbb A^k)^n = \mathbb A^{nk}.$$
Moreover, considering $Y \subseteq X$ an affine subvariety determined by $I=(f_1, \dots , f_r)$ its $k$-jet scheme is then the affine subscheme of $(\mathbb A^n)_k(\C)$ determined by the conditions $f_j(a^i) = 0$ modulo $t^k$ for $j= 1, \dots , r$.
\end{ex}

Next, we turn our attention to the notion of motivic classes of varieties. These classes are fundamental to defining the motivic measure introduced by Denef and Loeser in \cite{mea} for varieties over $\C$ since it takes value into the ring of motivic classes of varieties.

\begin{df}
The Grothendieck ring of varieties is defined as the quotient of the free abelian group generated on the set of isomorphism classes of varieties over $\C$ modulo the subgroup generated by the relations
    \begin{center}
    $[X] = [Z] + [X-Z]$ for any $X \in \operatorname{Var}_\C$ and $Z \subseteq X$ Zariski closed  
    \end{center}
and then equipped by the multiplication defined as
    \begin{center}
    $[X] [Y] = [X \times Y]$ for any $X,Y \in \operatorname{Var}_\C$.    
    \end{center}
The elements of this ring are called motivic classes and we denote the ring by $\mathbf{K_0}(\operatorname{Var}_\C)$ and by $\Li = [\mathbb{A}^1]$ the motivic class of the affine line.
\end{df}

We recall the following well-known and useful facts, from \cite[Proposition 5]{mot} and \cite[Lemma 2.9]{brid}.
\begin{prp} \label{fib}
Let $X,Y $ be varieties and let $f: X \to Y$ be a fibration of fiber $F$ which is locally trivial for the Zariski
topology of $Y$. Then, $[X] = [F][Y]$.
\end{prp}

\begin{prp} \label{geom}
Let $X,Y$ be varieties and let $f: X \to Y$ be a geometric bijection, that is a morphism which is also a bijection $f(\C) : X(\C) \to Y(\C)$ of $\C$-points. Then, $[X] = [Y]$.
\end{prp}

The construction of the motivic measure of constructible subsets of the arc scheme is based on the result \cite[Theorem 7.1]{mea} by Denef and Loeser, where the authors prove that the limit considered in the next definition exists.
As required by their result, we need to work with the localization $\mathbf{M_0} = \mathbf{K_0}(\operatorname{Var}_\C)[\Li^{-1}]$ as we need $\Li$ to be invertible and with the completion $\widehat{\mathbf{M}}_{\mathbf{0}}$ of $\mathbf{M_0}$ with respect to the filtration $F^k\mathbf{M_0}$ where $F^k\mathbf{M_0}$ denotes the subgroup generated by classes of varieties $[X]\mathbb{L}^{-i}$ with $\dim X -i \leq -k$. 

\begin{df} \label{def-measure}
Let $X$ be a variety.
A constructible subset in $X_k$ is a union of Zariski locally closed sets in $X_k$ and a constructible subset in the arc scheme $X_\infty$ is a subset of the form $A=\pi_k^{-1}(C_k)$ for $C_k \subset X_k$ constructible and some $k > 0$.
Then, the limit $$ \mu (A) = \lim_{ m\to \infty} \frac{[\pi_m(A)]}{\Li^{m \dim X}} $$ exists in $\widehat{\mathbf{M}}_{\mathbf{0}}$ and we call it the motivic measure of $A$.
\end{df}

\begin{rmk}
As observed in \cite[Definition-Proposition 3.2]{mea}, in the case of a smooth variety $X$ for $A=\pi_k^{-1}(C_k)$ a constructible set in the arc scheme $X_\infty$ the definition of motivic measure simplifies to $$\mu (A) = \displaystyle\frac{[\pi_k(A)]}{\Li^{k \dim X}}.$$
\end{rmk}

\begin{rmk}
As also discussed in the accompanying result \cite[Definition-Proposition 3.2]{mea} the definition of $\mu$ extends from constructible subsets to a $\sigma$-additive measure $\mu$ on the Boolean algebra of constructible subsets of the arc scheme.  
\end{rmk}

To conclude this section, we introduce the measure of certain special sets of arcs and their generating series, obtained as a parameteric motivic integral called Igusa zeta functions and taking values in the ring $\widehat{\mathbf{M}}_{\mathbf{0}}[[T]]$ of formal power series with coefficients in $\widehat{\mathbf{M}}_{\mathbf{0}}$. Motivic Igusa zeta functions were defined in \cite{DL}.
\begin{df} \label{ord} \label{def-val}
Let $W$ be a variety and let $V \subseteq W$ be a closed subvariety. 
We consider the integrable function $ord_V : W_\infty \setminus V_\infty \to \mathbb N$ returning the order of arcs along the non-zero coherent sheaf of ideals defining $V$. Concretely, if $\mathcal{I}$ denotes the latter and $p_\gamma = \gamma (\operatorname{Spec}(\C)) \in V$ then $$ord_V (\gamma) = \operatorname{inf} \{ val_t(\gamma^*(f)) \, | \, f \in \mathcal{I}_{p_\gamma} - \{ \mathbf{0} \}  \} $$ where $val_t$ denotes the $t$-adic valuation (or simply valuation). 
\end{df}
We often refer to $ord_{V}(\gamma)$ as the valuation of the arc.

\begin{rmk}
From now on, we will omit removing $V_\infty$ and just write $W_\infty$ since $\mu(V_\infty) =0$, see \cite[Chapter 7, Section 3, Subsection 3.3]{greenbook}.   
\end{rmk}

\begin{df} \label{def:contact-locus}
In the same setting as Definition \ref{def-val}, we define the $n$-th contact locus of $V$ as $$\Cont^n_V = ord^{-1}_V(n) \subseteq W_\infty$$ and its measure as $$\mu(\Cont^n_V) = \frac{[\pi_{n+1}(\Cont^n_V)]}{\Li^{(n+1) \dim X}} .$$
\end{df}
\begin{df} \label{def-igusa}
Let $W$ be a variety and let $V \subseteq W$ be a closed subvariety. Let $A \subseteq W_\infty$ be a constructible subset.
The motivic Igusa zeta function of $V$ over $A$ is defined as $$Z^A_V(T) = \displaystyle\int_{A} T^{ord_V} d\mu = \sum_{n = 0}^\infty\mu(\Cont^n_V)T^n \in \widehat{\mathbf{M}}_{\mathbf{0}}[[T]] .$$ 
\end{df}
The formal variable $T$ is usually written as $\Li^{-s}$ for $s$ a formal parameter, and we usually omit the upper $A$ when $A= \pi_1^{-1}(0)$.

\subsection{Embedded resolutions}
The study of embedded resolutions was extensively developed by Hironaka, in addition to many others, converging in his result regarding the existence of embedded resolutions of singular varieties over a field of characteristic zero. We refer the reader to \cite{resol} for any further details. 
Since our object of interest is curves, that is, the pair $(\mathbb A^2, C)$ for some curve $C$ we will assume the ambient variety to be smooth.

\begin{df} \label{def:strata}
Let $D$ be a divisor of a smooth variety, namely a closed subscheme of pure codimension $1$, and let $I$ be the set indexing of its irreducible components. We denote by $D_i$ the irreducible component indexed by $i \in I$.
For $J \subseteq I$ we denote with $$D_J =  \displaystyle\bigcap_{j \in J} D_j$$ and with $$D_J^o = D_J  \setminus \displaystyle\bigcup_{j \not\in J} D_j.$$
\end{df}

\begin{df}
Let $D$ be a divisor of a smooth variety. We say that $D$ has simple normal crossings if the irreducible components of $D$ are smooth and meet transversally.
\end{df}

\begin{df} \label{res}
Let $W$ be a variety and let $V \subseteq W$ be a closed subscheme.
Then, an embedded resolution of the pair $(W,V)$ is a pair $(Y,\rho)$ where $\rho: Y \rightarrow W$ is a proper birational morphism, $Y$ is smooth, and $\rho^{-1}(V)$ is a simple normal crossing divisor on $Y$.
\end{df}
\begin{df}
We call strict transform $\widetilde{V}$ of $V$ the component of $\rho^{-1}(V)$ that corresponds to the closure of the dense subset of $V$ where $\rho$ is an isomorphism.  
\end{df}
We usually denote by $I$ the indexing set of the irreducible components of $\rho^{-1}(V)$, including the strict transform unless otherwise specified (explicitly writing of $\widetilde{V}$ and following the notation $E_i, i \in I$ for the components. We denote their union by $$E = \displaystyle\bigcup_{i \in I } E_i.$$ 
We refer to a component $E_i$ as an exceptional divisor.

\begin{thm}[{\cite[Main Theorem II(N)]{resol}}]
Let $\C$ be a field of characteristic zero and $(W,V)$ a variety and closed subvariety respectively over $\C$. Then, there exists an embedded resolution $(Y, \rho)$ of the pair such that it is an isomorphism outside of $W_s \cup V$ where $W_s$ denotes the singular locus of $W$. Moreover, the morphism $\rho$ is a composition of blow-ups with smooth centers. 
\end{thm}

Utilizing the transversality condition, we can deduce certain numerical information of the singularity as presented in \cite[Chapter 7, Theorem 3.3.4]{greenbook}. Before that, we recall the notion of multiplicity of a non-zero coherent sheaf of ideals at a point.
\begin{df}
We recall that given a variety $X$ and a point $p \in X$, the multiplicity of a non-zero coherent sheaf of ideals $\mathcal{I}$ on $X$ at $p$ is defined as $$ \operatorname{max} \{ \mu \, | \, \mathcal{I}_p \subseteq m_p^\mu \} $$ where $m_p \subseteq \mathcal{O}_{X, p}$ denotes the maximal ideal of $p$.
\end{df}

\begin{df} \label{ddef-data} \label{mult}
Let $\rho: (Y, E) \rightarrow (W,V)$ be an embedded resolution as in Definition \ref{res} and let $\mathcal{I} \subseteq \mathcal{O}_W$ be the non-zero coherent sheaf of ideals on the ambient variety $W$ determining the subvariety $V$.
For every $i \in I$ we denote by $N_i$ the multiplicity of the pullback ideal sheaf $\mathcal{I} \, \mathcal{O}_Y \subseteq \mathcal{O}_Y$ along any smooth point of $E_i$ and by $\nu_i$ the multiplicity of the relative canonical divisor $K_{Y | W}$ along any smooth point of $E_i$. These are respectively called multiplicity and discrepancy of $E_i$.
\end{df}

Finally, we present a classic example of an embedded resolution, whose details can be found in \cite[Example 3.1.9]{greenbook}.
\begin{ex} \label{ex-cusp}
We briefly recall the resolution of singularities of the cusp $f = x^3 - y^2 \in \C[x, y]$. One can check that three blowups are needed, thus obtaining three exceptional divisors in addition to the strict transform of the curve for which the multiplicities and discrepancies are $(N_i, \nu_i)_{i=1,2,3} = (2, 1),(3, 2),(6, 4)$.
\end{ex}

We conclude this section by presenting a formula for $[\pi_{n+1}(\Cont^n_V)]$, hence $\mu(\Cont^n_V)$, in terms of an embedded resolution of $(W,V)$, this highlighting a connection between invariants of singularities and motivic integration, which was also our original reason of interest, as explained in Section \ref{intro}. We refer to it as Denef's formula, as it was proven in \cite[Proposition 3.2.1]{DL}.
\begin{thm} \label{denef-contact}
Let $\rho: (Y, E) \rightarrow (W,V)$ be an embedded resolution as in Definition \ref{res}. We have
$$[\pi_{n}(\Cont^n_V)] = 
\Li^{n \dim W}
\sum_{J \subseteq I, J \neq \emptyset}
(\Li-1)^{|J|-1}[E_J^o]
\left( \sum_{k_j \geq 1, \, j \in J, \, \sum k_jN_j = n}
\Li^{-\sum_{j \in J} k_j \nu_j} \right)$$
for every $n \geq 1$.
\end{thm}

\subsection{Puiseux parameterizations and unibranch curves} \label{sec:puiseux}

Unless otherwise stated, $(C,0)$ will always denote a curve locally defined as the zero locus of $f \in \mathbb \C[[x,y]]$, $f(0)=0$.
We denote by $f_d$ the $d$-th degree homogeneous component of $f$, and the smallest $d >0$ such that $f_d \neq 0$ is called multiplicity of the singularity and denoted by $mult_C(0) = mult_C$.

\begin{df} \label{def-unibranchcurve}
Let $(C, 0)$ be a curve and let $f \in \C[[x,y]], f=\prod_i f_i $ be the equation locally defining it and its factorization in $\C[[x,y]]$ in irreducible factors allowing repetitions. We call branches of $C$ the curves corresponding to each $f_i$, and we say that $C$ is unibranch when it has only one branch i.e. $f$ is irreducible in $\C[[x,y]]$.
\end{df}

From now on, we will always work in the unibranch case.
We now recall some details about Puiseux parameterizations of unibranch curves, and we refer the reader to \cite{Brieskorn} or \cite{Wall} for any further details.

\begin{df} \label{def-powerseries}
We use the notation $\C[[t]]$ for the domain of formal power series with coefficients in the field $\C$ and indeterminate $t$, and we write its general element as $$p = p(t) = \displaystyle\sum_{n = 0}^\infty a_n t^n.$$
Moreover, we denote by $\C[[t]]_0$ the domain of formal power series centered at $0$, that is $\C[[t]]_0 = \C[[t]] \setminus \C[[t]]^\times$. This equivalently corresponds to formal power series with $$a_0 =0.$$
\end{df}
\begin{df} \label{def-t-valuation}
Following the notations of Definition \ref{def-powerseries}, we say that $p \in \C[[t]]$ has valuation $k$ if $a_k \neq 0$ is the smallest non-zero coefficient. We denote this by $$k = val_t(p).$$
\end{df}
\begin{rmk}
We observe that $\C[[t]]_0 = \{ p \in \C[[t]] \, | \, val_t(p)>0 \}$.
\end{rmk}

\begin{df} \label{def-params}
Let $p,q \in \C[[t]]_0$ as in Definition \ref{def-powerseries} and let $(C,0)$ be a unibranch curve as in Definition \ref{def-unibranchcurve}. We say that the pair $(p,q)$ is a parameterization of $C$ if $$f(p,q)=0$$ in $\C[[t]]$. 
\end{df}
\begin{df} \label{def:smooth-reps}
We say that two pairs $(p,q)$, $(p',q')$ are equivalent parameterizations when they are parameterizations for the same curve $C$, that is, there exists $\alpha \in \C[[t]]_0$ with $val_t(\alpha)=1$ such that $$p'(t)=p(\alpha(t)), q'(t)=q(\alpha(t))  .$$ 
\end{df}
\begin{rmk} \label{rmk-degenerate-par}
We say that a parameterization $(p',q')$ is degenerate if there exist another parameterization $(p,q)$ and $\alpha \in \C[[t]]_0$ with $val_t(\alpha)>1$ such that $p'(t)=p(\alpha(t)), q'(t)=q(\alpha(t)) $. 
Viewing $\C[[t]]$ as $(\mathbb A^1)_\infty$ and following \cite[Definition, pg. 5]{sympow-gorsky}, we may let the monoid of $\C$-algebra endomorphisms of $\C[[t]]$ of the form $t \mapsto \alpha$, for $\alpha \in \C[[t]]_0$, act on $\C[[t]]^2 \cong (\mathbb A^2)_\infty = \Lo$. The invertible elements of this monoid are exactly those with $val_t(\alpha)=1$, and they form the group of smooth reparameterizations of Definition \ref{def:smooth-reps}, while the endomorphisms with $val_t(\alpha)>1$ lead to orbits in $\Lo$ of motivic measure zero.
For this reason, and since degenerate parameterizations do not arise from the equivalence of Definition \ref{def:smooth-reps}, we will always assume our parameterizations to be non-degenerate.
\end{rmk}
\begin{rmk} 
We mention here that equivalence classes of parameterizations are also called branches since they all correspond to the same geometric locus, that is, the same image in $\mathbb A^2$. Therefore, the branch of the unibranch curve $C$ refers to the zero locus determined by $f$ as well as the set of all parameterizations of $C$.
\end{rmk}
\begin{rmk}
We observe that composing by $\alpha$ in Definition \ref{def:smooth-reps} induces an automorphism of $\C[[t]]$ that preserves $val_t$. This implies that all the parameterizations in a given class have the same valuations.
\end{rmk}

We now recall a classic result by Puiseux, guaranteeing the existence of parameterizations of the curve defined by $f(x,y)=0$. Finding such a parameterization corresponds to solving the equation $f(x,y)=0$ for one of the variables in terms of the other, that is, to finding a root of $f$, regarded as a polynomial in $y$ with coefficients in $\C[[x]]$, in the field of Puiseux series $P = \displaystyle\cup_{k=1}^\infty \C((t^\frac{1}{k}))$. Puiseux's theorem ensures the existence of such a root by constructively proving that $P$ is algebraically closed. The same result and proof can be extended to $f \in \C[[x,y]]$, namely power series instead of polynomials, as pointed out in \cite[Section 2.1]{Wall}.

\begin{thm}[\cite{pui}] \label{thm-puiseux}
Let $\C((t))$ denote the quotient field of $\C[[t]]$ and let $P = \displaystyle\cup_{k=1}^\infty \C((t^\frac{1}{k}))$. Then, $P$ is an algebraically closed field.
\end{thm}

\begin{rmk}[{\cite[Pg. 406]{Brieskorn}}] \label{rmk}
From the proof of Theorem \ref{thm-puiseux} it also follows that in the equivalence class of Definition \ref{def-shape-of-param} there exists a representative of the form $$(p_C, q_C) = \left(t^n, \sum_{m \geq n} a_m t^m \right)$$ with $n = mult_C$.
\end{rmk}
\begin{df} \label{def-repams} \label{def-shape-of-param}  \label{def:puiseux-param}
We refer to the parameterization $(p_C, q_C)$ of Remark \ref{rmk} as a Puiseux parameterization of $(C,0)$.  
\end{df}

\begin{df} \label{def:pq-curve}
Let $(C,0)$ be a unibranch curve. We say that $C$ is a $(p,q)$-curve when $(p_C , q_C) = (t^p, t^q)$ for some $p,q >0$ with $gcd(p,q)=1$.
\end{df}

We also recall this useful fact, which allows us to compute the intersection number of two curves using any of their equivalent parameterizations.
\begin{lemma}[{\cite[Lemma 1.2.1]{Wall}}] \label{back:intersection-numb} \label{lemma-intersection-num}
Let $(p,q)$ be a parameterization of $C$ as in Definition \ref{def-params} and let $(D,0)$ be another unibranch curve locally defined by $g \in \C[[x,y]]$, $g(0)=0$. Then, the intersection number of the two curves can be computed as $$C.D = val_t(g(p,q))  .$$  
\end{lemma}

\subsubsection{Properties of the completion of the local coordinate ring of a unibranch curve}
Let $(C,0)$ be a unibranch curve as in Definition \ref{def-unibranchcurve} and let $(p_C,q_C)$ denote its Puiseux parameterization as in Definition \ref{def:puiseux-param}. We denote by $\ring$ the (completion of the) local ring of the curve, that is $$\ring \cong \C[[x,y]] / (f) \cong \C[[p_C, q_C]]. $$

We need to keep track of the valuations appearing in $\ring$ since they will be crucial in later sections to get to our central results. To do so, we recall the notion of (numerical) semigroup of $C$.
\begin{df} \label{def-semigroup}
We denote by $$\Gamma_C = \{ val_t(g) \, | \, g \in \ring   \}  $$
and we refer to it as the semigroup of the curve. We simply denote it by $\Gamma$ when it is clear which curve $C$ we refer to. We denote by $$H = \mathbb N \setminus \Gamma$$ the set of holes of the semigroup, and we refer to $$|H| = \delta$$ as the delta invariant of the curve.
\end{df}

We recall that any semigroup, considered as a submonoid of $\mathbb N$, admits a unique minimal system of generators.
\begin{thm}[{\cite[Theorem 2.7]{semig}}] \label{thm:semigroup-gens}
Every semigroup admits a unique minimal system of generators. This minimal system of generators is finite.
\end{thm}

We are interested in a particular class of semigroups, called semigroup ideals, as they will be the semigroups of valuations appearing from $\ring$ in the Hilbert scheme.
\begin{df}
We say that $S \subseteq \Gamma$ is a semigroup ideal if $S + \Gamma \subseteq S$.
\end{df}
\begin{ex}
Let $J \subseteq \ring$ be an ideal and let $\Delta = val_t(J) = \{ val_t(g) \, | \, g \in J   \}$ be the semigroup of the valuations appearing in $J$. It is clear that $\Delta$ is a semigroup ideal. However, the converse is false, as pointed out in \cite[Remark, pg. 13]{os}.
\end{ex}
\begin{df}
Let $\Delta \subseteq \Gamma$ be a semigroup ideal. We denote by $c(\Delta) = min\{ j \in \Delta \, | \, j + \mathbb N \subseteq \Delta \}$ the element from where there are no holes anymore, and we refer to it as conductor of $\Delta$.
\end{df}

Many elements could have the same valuation in $\ring$. However, each valuation $n \in \Gamma$ can be achieved by a special element, as proved in \cite{os}.
\begin{lemma}[{\cite[Lemma 25]{os}}] \label{lemma-os-minimalelem}
Let $n \in \Gamma$. Then, there exists a unique $\tau_n \in \ring$ of the form $$\tau_n = t^n + \displaystyle\sum_{i \in H, i >n} c_{n,i} t^i .$$   
\end{lemma}

\subsection{The Hilbert scheme of unibranch curves}

The Hilbert scheme of points on $C$ is defined as the moduli space parameterizing zero-dimensional subschemes of $C$ with fixed length, originally introduced in \cite{EGA} for any variety $X$. We recall its definition starting with the construction as a functor of points and then present its representability in the category $\operatorname{Sch}_\C$ of schemes.

\begin{df} \label{def-fun}
Let $X$ be a scheme. For any $n \geq 1$ we consider the functor $H_X^n : \operatorname{Sch}_\C^{op} \to \operatorname{Set}$ 
defined on objects of $\operatorname{Sch}_\C$ as 
$$ S \mapsto H_X^n(S) = \{ Z \subseteq X \times S \, | \,  p|_Z : Z \to S \text{ finite, flat and surjective of degree } n \}$$ and on morphisms given $f: T \to S$ morphism in $\operatorname{Sch}_\C$ as 
\begin{center}
    $H_X^n(f) : H_X^n(S) \to H_X^n(T), \,  Z \mapsto (id_X , f)^{-1} (Z) = \,$ pullback with respect to $(id_X, f) : X \times T \to X \times S$.
\end{center}
\vspace{2mm}       
\begin{center}
\begin{tikzcd}
Z\times_S T \subseteq X \times T \arrow[d, "p_1"'] \arrow[r, "p_2"] & T \arrow[d, "f"] \\
Z \subseteq X \times S \arrow[r, "p|_Z"] & S    
\end{tikzcd}
\begin{tikzcd}
{(x,t)} \arrow[d, maps to] \arrow[r, maps to] & t \arrow[d, maps to]\\
{z=(x,s)} \arrow[r, maps to] & s=f(t)
\end{tikzcd}
\end{center}
\end{df}

The representability of the so-called Hilbert functor $H_X^n$ has been proved in \cite[Theorem 3.2]{EGA}. This is a key fact, as our primary focus will rely on the scheme representing it.

\begin{df} \label{def-hilb}
We denote by $\Hilb^n(X)$ the scheme representing the functor $H_X^n$ and we call it the Hilbert scheme of points of $X$ of length $n$.
The universal subscheme is the subscheme $\mathcal{Z} \in H_X^n(\Hilb^n(X))$ corresponding to $id_{\Hilb^n(X)}$ under representability. 
\end{df}

\begin{rmk}
Using the representability we have a correspondence $$pt \mapsto \mathcal{Z} \displaystyle\cap ( X \times \{ pt \})$$ between closed points of $\Hilb^n(X)$ and zero-dimensional subschemes $Z \subseteq X$ with finite length equal to $\dim \mathcal{O}_Z = n$.
One can view these subschemes as points on the Hilbert scheme.
\end{rmk}

\begin{df}
    Given a zero-dimensional finite-length subscheme $Z \in \Hilb^n(X)$ we call its support the set of points $|Z| = \{ x_1 , \dots , x_k\}$ where the scheme is geometrically supported, and we define the cluster associated to $Z$ as the $0$-cycle $[Z] = \displaystyle\displaystyle\sum_{x_i \in |Z|} \dim \mathcal{O}_{Z, x_i} \, x_i$.
\end{df}

Constraining the support, we can define punctual schemes and {the punctual Hilbert scheme of points} as the ones with support concentrated in one point of higher multiplicity. 

\begin{df} \label{d-h}
Let $X$ be a scheme and let $p \in X$. We denote by $$\Hilb_p^n(X) = \{ Z \in \Hilb^n(X) \, | \, |Z| = \{ p \} \}$$ the closed subscheme of $\Hilb^n(X)$ made of the zero-dimensional subschemes of length $n$ that are supported entirely at the point $p$.
\end{df}

From now on, we will omit the word ``punctual'' and we will always refer to the Hilbert scheme as the one of Definition \ref{d-h}.

Returning to our local case of interest, which is unibranch curves, the Hilbert scheme can be described in an even clearer way. Let $(C,0)$ be a unibranch curve and let $\ring$ be (the completion of) its local ring of coordinates, as defined in Section \ref{sec:puiseux}.

\begin{ex} \label{ex-ide}
Given a unibranch curve $(C,0)$ the Hilbert scheme takes the form 
 $$\Hilb^n_0(C)=\{ I \subseteq  \ring \text{ ideal} \, | \dim_\C\ring / I = n\} = \{ J \subseteq \C[[x,y]] \text{ ideal} \, | \, f \in J \, , \, J \subseteq (x,y) \, , \, \dim_\C\C[[x,y]] / J = n \}.$$
\end{ex}

In this paper, we will be interested in points of the Hilbert scheme with a given number of minimal generators. We recall the following map, that appears in \cite[Conjecture 2']{os}.
\begin{df} \label{def:number-gens}
Let $m : \Hilb^n_0(C) \to \mathbb N_{>0}$ be the map defined as 
$$ I \mapsto \dim_\C I/m_0I$$ where $m_0$ denotes the maximal ideal of $\ring$. Equivalently, by Nakayama's lemma we have that $m(I)$ corresponds to the minimal number of generators of $I$ in $\ring$.
\end{df}

\begin{df}
We denote by $\KHilb{n} \subseteq \Hilb^n_0(C)$ the closed subscheme of schemes with $k$ minimal generators i.e. $\KHilb{n}=m^{-1}(k)$ or more explicitly
$$\Hilb_{0, k}^n(C)= \{ I \in \Hilb^n_0(C) \, | \, \dim_\C I/m_0I = k \} .$$ In particular, when $k=1$ we refer to $\PHilb{n}$ as principal Hilbert scheme or one-generator locus, and to its points as principal schemes.  
\end{df}

\begin{rmk}
As observed in \cite[Proof of Theorem 13]{os}, the minimal number of generators of a point in $\Hilb_0(C)$ is bounded by $mult_C$: $\KHilb{n} = \emptyset$ for $k > mult_C$.
\end{rmk}

\subsubsection{Valuation semigroups, syzygies and Hilbert schemes}
From \cite{os} and \cite{ors} we recall the following key results, which allow us to control the generators and semigroup ideals appearing in principal or fixed-generator subschemes. These results will be fundamental in Section \ref{section:length} to control the motivic class of $\KHilb{n}$.

\begin{thm}[{\cite[Definition 26, Theorem 27]{os}}] \label{os-elimination-allowed}
Let $\Delta \subseteq \Gamma$ be a semigroup ideal and let $\nu_1, \dots , \nu_l$ denote its generators as a semigroup ideal $$\Delta = \displaystyle\coprod_{i = 1}^l \{ \nu_i + \Gamma  \} .$$
We denote $$\Sigma_i = \Gamma_{> \nu_i} \setminus \Delta$$ and $$f_i =  \tau_{\nu_i} + \displaystyle\sum_{j \in \Sigma_i} \lambda_j^i \tau_j  $$ 
where $\tau_{\nu_i},\tau_j$ are defined as in Lemma \ref{lemma-os-minimalelem}. Let $J_\Delta = (f_1, \dots , f_l)$ and let $$U_\Delta \subset \operatorname{Spec}\C[\lambda_j^i \, | \, j \in \Sigma_i, i= 1, \dots , l] , \, \, U_\Delta = \{ s \in \operatorname{Spec}\C[\lambda_j^i \, | \, j \in \Sigma_i, i= 1, \dots , l] \, | \, 
\dim_\C \ring / (J_\Delta|_s) = |\Gamma \setminus \Delta|  \} $$
where the notation $e|_s$ for $e \in U_\Delta$ means taking $\lambda_j^i = s_{i,j} \in \C$ for every $i,j$ (and for an ideal, it means evaluating all generators).
Let $\Hilb^\Delta_0(C)$ denote the subscheme of schemes $I$ such that $val_t(I) = \Delta$. Then, there exists a bijective morphism $$ U_\Delta \to \Hilb^\Delta_0(C), \, s \mapsto J_\Delta|_s .$$
\end{thm}

\begin{rmk} \label{rmk:os-elimination-content}
Theorem \ref{os-elimination-allowed} contains two pieces of information. On the one hand, it provides a normal form for the generators of an ideal with prescribed semigroup: every such ideal admits generators $f_i$ whose expansion in the basis $\{ \tau_j \}_{j \in \Gamma}$ of Lemma \ref{lemma-os-minimalelem} only involves the leading term $\tau_{\nu_i}$ and the terms $\tau_j$ with $j \in \Sigma_i$, all the other coefficients being eliminated by row reduction. On the other hand, and this is the part we will use, it identifies the locus $U_\Delta$ of coefficients producing an ideal of the expected colength with $\Hilb^\Delta_0(C)$, so that motivic classes can be computed in the affine coordinates $\lambda^i_j$. We mention here that $U_\Delta$ is in general a proper subvariety of the affine space $\operatorname{Spec}\C[\lambda^i_j]$, since the coefficients have to be constrained for the generated ideal to have valuation set exactly $\Delta$, and these constraints are made explicit in Theorem \ref{thm:syz-ors}.
\end{rmk}

\begin{rmk} \label{simplify-reps}
In particular, Theorem \ref{os-elimination-allowed} takes explicit shape as follows for principal schemes.
By definition, we have that $\PHilb{n} = \Hilb^\Delta_0(C)$ for $\Delta = n + \Gamma = \nu_1 + \Gamma$. 
This implies that $\PHilb{n} = U_\Delta$ and all principal ideals are consequently generated by elements 
$$f_n = \tau_n + \displaystyle\sum_{h \in H \, | \, n+h \in \Gamma} \lambda_h^n \tau_{n+h} $$
for $\lambda_h^n \in \C$ since $$\{n + H \} \displaystyle\cap \Gamma = \{n + H \} \displaystyle\cap \Gamma_{ >n} = \Gamma_{>n} \setminus \{n + \Gamma \} = \Sigma_1 .$$
Therefore, we can write 
$$\PHilb{n} = \{ (f_n) \, | \, \lambda_h^n \in \C \}  $$
and the bijection between $\PHilb{n}$ and $\mathbb A^{| \{ n + H \} \cap \Gamma|}$ immediately follows. It extends to an isomorphism, following \cite[Theorem 1]{piontk}.
\end{rmk}

Finally, we want to extract explicit conditions on the variables $\lambda_j^i$ of Theorem \ref{os-elimination-allowed} ensuring that the ideal they generate has its semigroup equal to a given semigroup ideal $\Delta \subseteq \Gamma$.
First, we start by recalling the following results from \cite{ors} that allow us to simplify generators of ideals in $\ring$.
\begin{df}[{\cite[Section 3.1]{ors}}] \label{def-syzy}
We denote by $\C[\Gamma] = \C[t^i \, | \, i \in \Gamma]$ and we choose a basis $\{ \phi_i \}_{i \in \Gamma}$ of $\ring$ that is also a basis of $\C[\Gamma]$. We denote it by
$$ \phi_i = t^i + \sum_{j>i, \,j\in \Gamma} a_{ij}t^j .$$
\end{df}
\begin{lemma}[{\cite[Section 3.1]{ors}}]  \label{def:ors-generators}
Let $J \subseteq \ring$ be an ideal such that $val_t(J) = \Delta$ and let $i \in \Delta$. Then, there exists a unique $f_i \in J$ of the form $$f_i = \phi_i + \sum_{k \in \Gamma_{>i} \setminus \Delta} \lambda_k^i \phi_k $$
with $\lambda_k^i \in \C$. 
\end{lemma}
\begin{rmk}
Both Lemma \ref{def:ors-generators} and Theorem \ref{os-elimination-allowed} provide normal forms for the generators of an ideal with prescribed semigroup, with respect to two different bases of $\ring$: the basis $\{ \tau_i \}_{i \in \Gamma}$ of Lemma \ref{lemma-os-minimalelem}, whose tails are supported on the holes $H$, and the basis $\{ \phi_i \}_{i \in \Gamma}$ of Definition \ref{def-syzy}, whose tails are supported on $\Gamma$. The two bases coincide for $(p,q)$-curves, where $\ring = \C[\Gamma]$ is monomial and $\tau_i = \phi_i = t^i$. As in \cite{ors}, we will work with the second one: the expansions of the products $\phi_i \phi_j$ in the basis $\{ \phi_i \}$ provide the structure constants appearing in the syzygy conditions of Theorem \ref{thm:syz-ors}, which allows a simpler treatment.
\end{rmk}

We can now define conditions on the coefficients $\lambda_j^i$ of a minimal set of generators $f_i$ of an ideal, following the notation introduced in Lemma \ref{def:ors-generators}, that determine the semigroup.
\begin{df}[{\cite[Equation (7)]{ors}}] \label{ors-phi-gens}
    Let $\Delta \subseteq \Gamma$ be a semigroup ideal and let $\nu_1, \dots , \nu_l$ be its generators as a semigroup ideal, so that $\Delta = \coprod_{i=1}^l \{ \nu_i + \Gamma \}$ as in Theorem \ref{os-elimination-allowed}. Let $\gamma_1, \dots , \gamma_s$ be the minimal system of generators of the semigroup $\Gamma$ of Theorem \ref{thm:semigroup-gens}, so that $\C[\Gamma] = \C[t^{\gamma_1}, \dots , t^{\gamma_s}]$.
    Then, let $I_\Delta = (t^{\nu_1}, \dots , t^{\nu_l}) \subseteq \C[\Gamma] = \C[t^{\gamma_1}, \dots , t^{\gamma_s}]$.
    We consider the surjection 
    $$G : \C[\Gamma]^{\oplus l} \to I_\Delta, \, (0, \dots ,g_i, \dots ,0) \mapsto g_it^{\nu_i}$$
    and the free resolution 
    $$ \dots \to \C[\Gamma]^{\oplus m_\Delta} \to \C[\Gamma]^{\oplus l} \to I_\Delta \to 0$$
    that $G$ is part of as per Hilbert's syzygy theorem \cite[Theorem 1.1]{hilbert}.
    We denote by $S$ the map $ \C[\Gamma]^{\oplus m_\Delta} \to \C[\Gamma]^{\oplus l}$ of the resolution, by $m_\Delta$ the number of syzygies of $\Delta$, and we refer to $\operatorname{Im}(S) = \operatorname{ker}(G)$ as syzygies of $\Delta$.
\end{df}
\begin{df}[{\cite[Equation (7)]{ors}}] \label{def:syzygy-matrix}
    The map $S$ is represented by a $m_\Delta \times l$-matrix with entries in $\C[\Gamma]$. As proved in  \cite[Lemma 4]{piontk}, we can write $$(S)_{ij} = u^j_i \, t^{\sigma_i - \nu_j}$$ for some constants $u^j_i \in \C$ and some powers $\sigma_i \in \Gamma$. We call the powers $\sigma_i$ orders of the syzygies.
\end{df}
 The order $\sigma_i$ is the degree in which the $i$-th syzygy is supported, that is $\sigma_i = \nu_j + val_t((S)_{ij})$ for every $j$ such that $u^j_i \neq 0$. These are the same constants $u^j_i$ appearing in Theorem \ref{thm:syz-ors}.
\begin{df}[{\cite[pg. 10, pg. 12]{ors}}]  \label{notation}
We denote by
$$ Syz =
\operatorname{Spec}(\C[x^{\sigma_i}_{is} \, | \,  c(\Delta) > s > \sigma_i \, , i=1 ,\dots, m_\Delta ])$$
the affine space
of dimension $\sum_{i = 1}^{m_\Delta} (c(\Delta) - \sigma_i -1)$.
We denote by $$Gen = \operatorname{Spec}(\C[\lambda_k^{\nu_i} \, | \,  k \in \Gamma_{> \nu_i} \setminus \Delta \, , i=1 ,\dots, l ])$$ the affine space of dimension $\sum_{i = 1}^{l} |\Gamma_{> \nu_i} \setminus \Delta|$ of the coefficients of the generators $f_{\nu_i}$ of Lemma \ref{def:ors-generators}. 
\end{df}
\begin{thm}[{\cite[Proposition 12]{ors}}] \label{thm:syz-ors}
Let $\lambda \in Gen$ and let $J_\lambda$ be the corresponding ideal. Then, $val_t(J_\lambda) = \Delta$ if and only if there exists $p \in Syz$ such that the equations
$$ \sum_j \left(
u^j_i \phi_{\sigma_i - \nu_j} \phi_{\nu_j}  +
\sum_{c(\Delta) > s > \sigma_i } x^{\sigma_i}_{is} \phi_{s - \nu_j} \phi_{\nu_j} 
+
\sum_{k \in \Gamma_{> \nu_j} \setminus \Delta} u^j_i \lambda_k^{\nu_j} \phi_{\sigma_i - \nu_j} \phi_{k} 
+
H.O.T.
\right) = o(t^{c(\Delta)})$$
for $i = 1, \dots , m_\Delta$ are satisfied. Moreover, if such $p \in Syz$ exists, it is unique. 
Here the index $j$ runs over $1, \dots , l$, the constants $u^j_i \in \C$ are the entries of the syzygy matrix of Definition \ref{def:syzygy-matrix}, the $\phi$'s are the basis elements of Definition \ref{def-syzy}, the variables $x^{\sigma_i}_{is}$ are the coordinates of $Syz$ and the $\lambda^{\nu_j}_k$ the coordinates of $Gen$ as in Definition \ref{notation}, the integer $c(\Delta)$ is the conductor of $\Delta$, and $H.O.T.$ denotes the higher-order terms of \cite[Proposition 12]{ors}.
\end{thm}

\subsubsection{Mather's contact equivalence for points of Hilbert schemes} \label{sec:Mather}
To conclude the background section we introduce the notion of $\mathcal{C}$-equivalence or contact equivalence, presented by Mather in his work \cite{Mather}. It will be useful to control all the different representatives of any given ideal in $\ring$.

\begin{df}
We denote by $\mathcal{C}$ the group, under composition of map-germs, of $C^\infty$-map-germs $H : (\C^2 \times \C^k , 0 ) \to (\C^2 \times \C^k, 0)$ such that the diagram
\begin{center}
    \begin{tikzcd}
                                                       & {(\C^2 \times \C^k,0)} \arrow[rd, "p_1", two heads] \arrow[dd, "H"] &            \\
{(\C^2,0)} \arrow[ru, "i", hook] \arrow[rd, "i", hook] &                                                                   & {(\C^2,0)} \\
                                                       & {(\C^2 \times \C^k,0)} \arrow[ru, "p_1", two heads]                 &           
\end{tikzcd}
\end{center}
is commutative, where $i$ is the inclusion $z \mapsto (z,0)$ and $p_1$ the projection onto the first factor. This group acts on $C^\infty$-map germs $g : (\C^2,0) \to (\C^k,0)$ as $$H \bullet g = p_2(H(id,g))  $$ where $p_2$ denotes the projection onto the second factor.
\end{df}

\begin{df} \label{def-matherequiv} \label{def:Mather-equiv}
Let $g_1, g_2 : (\C^2,0) \to (\C^k,0)$ be two $C^\infty$-map-germs and let $g_i = (g_i^1, \dots , g_i^k)$, with $g_i^j \in \C[[x,y]]$ and $g_i^j (0)=0$, be their components in the coordinates $x,y$ of $\C^2$, for $j=1, \dots,k$.
We say that $g_1$ is (contact, $\mathcal{C}$-) equivalent to $g_2$ if they are in the same $\mathcal{C}$-orbit.
\end{df}

\begin{thm}[{\cite[Section 2.3]{Mather}}] \label{thm-matherequiv} 
Let $g_1, g_2 : (\C^2,0) \to (\C^k,0)$ be two $C^\infty$-map-germs and let $g_i = (g_i^1, \dots , g_i^k)$, with $g_i^j \in \C[[x,y]]$ and $g_i^j (0)=0$, be their components in the coordinates $x,y$ of $\C^2$, for $j=1, \dots,k$.
Then, $g_1$ is equivalent to $g_2$ if and only if $$(g_1^1, \dots , g_1^k) = (g_2^1, \dots , g_2^k)$$ as ideals in $\C[[x,y]]$.
Moreover, this is also equivalent to the existence of an invertible matrix $U =(a_j^i)_{ij} \in GL_k(\C[[x,y]])$ such that $$ g_1^i = \sum_j a_j^i g_2^j$$ for every $i = 1, \dots ,k$.
\end{thm}

\begin{rmk}
Since we will work with ideals in $\ring = \C[[p_C, q_C]] = \C[[x,y]]/(f)$ we will rather use their notation in just one parameter, substituting the two variables $x,y$ as $p_C, q_C \in \C[[t]]$ respectively.
\end{rmk}

In particular, when two ideals are equal, it means that as points of the Hilbert scheme, they are the same.


\section{Motivic classes of fixed-generators Hilbert schemes of unibranch curves} \label{sec:motivic-classes}

To discuss motivic classes of $\KHilb{n}$ in general, we need to first simplify $\KHilb{n}$ as much as possible. The goal of this section is to construct some special representatives for the points of $\KHilb{n}$, that will lead to a partition of $\KHilb{n}$ into components whose motivic classes are easier to compute.
We start with some notations we will use in the rest of the section.
\begin{df} \label{general-notation} \label{n-uple}
Let $(C,0)$ be a unibranch curve.
Let $I \in \KHilb{n}$ be an ideal and let $g_1, \dots , g_k$ denote a minimal set of generators.
When considered in $\ring$ we can write 
$$g_i = \tau_{n_i} + \displaystyle\sum_{n_i < n \in \Gamma} \eta^i_n \tau_n$$
where $\tau_{n_i}, \tau_n$ are defined as in Lemma \ref{lemma-os-minimalelem}, $\eta^i_n \in \C$ and $$n_i = C_{g_i} . C$$ is the intersection number of $C$ and the curve locally given by $g_i$ by Lemma \ref{lemma-intersection-num}. We assume $n_1 \leq \dots \leq n_k$ and we denote the $k$-uple they form by $\underline{n}=(n_1, \dots , n_k)$.
\end{df}

The first thing we need to observe is that we can simplify the generators of the ideal in $\ring$. Following the same reasoning as in Remark \ref{simplify-reps} and the same notations as in Theorem \ref{os-elimination-allowed},
we can consider the following simplified generators.
\begin{rmk} \label{principal-generators-reduced}
Let $I$ be as in Definition \ref{general-notation}. Then, taking the semigroup ideal $\coprod_i \{n_i + \Gamma \} \subseteq val_t(I)$ in Theorem \ref{os-elimination-allowed} we obtain the representative
$$I = (f_1, \dots , f_k)$$ with 
$$f_i = \tau_{n_i} + \displaystyle\sum_{h \in H \, | \, n_i+h \in \Gamma} \eta^i_{n_i+h} \tau_{n_i+h} .$$ 
For now on we use this minimal set of generators, with a shift of the coefficients of $f_i$ and denoting them as $\lambda_h^i = \eta^i_{n_i+h}$ for every $i = 1, \dots , k$. We will also often omit the upper index $i$ to ease the notation.
\end{rmk}

\subsection{Intersection numbers and $\KHilb{n}$}

In this section, we use the intersection numbers $n_i$ of Definition \ref{general-notation} and Remark \ref{principal-generators-reduced} to create the special representatives for points of $\KHilb{n}$ mentioned just before. This will also get us to a first, rough partition of $\KHilb{n}$.
To begin, we need to introduce a notion of independence for the $n_i$. The main idea behind this notion is that we do not want the leading terms of generators in $\ring$ to cancel. This means we do not want the $k$-uple (which keeps track of the leading terms) to change when considering different representatives of the same ideal.

\begin{df} \label{def:independent-int-numbers}
Let $(C,0)$ be a unibranch curve. Let $I \in \KHilb{n}$ and let $g_1 , \dots , g_k$ be a minimal set of generators. Let $val_t(g_i) = n_i$ and let $\Delta_{g_i} = val_t(g_1, \dots, g_{i-1} ,g_{i+1}, \dots , g_k)$ for every $i=1, \dots ,k$. Then, we say that the valuations of this set of generators form an independent $k$-uple when 
$$n_i \not\in \Delta_{g_i}$$
for every $i=1, \dots ,k$.
\end{df}

We now observe that for any given ideal $I \in \KHilb{n}$, out of all representatives of $I$ and the corresponding minimal sets of generators and $k$-uples, there exists only one independent $k$-uple.

\begin{thm} \label{prop:change-rep} \label{uniqueness-intersection}
Let $I \in \KHilb{n}$. Then, there exist a minimal set of generators $g_1 , \dots , g_k$ such that the valuations $val_t(g_i) = n_i$ form an independent $k$-uple in the sense of Definition \ref{def:independent-int-numbers}, and this is the only possible independent $k$-uple of valuations arising from minimal sets of generators of $I$. 
\end{thm}

\begin{proof}
Let $\Delta = val_t(I)$.
Given two $k$-uples of valuations $\underline{n}, \underline{n'}$ we say that $\underline{n} < \underline{n'}$ when there exists a $0 \leq j \leq k$ such that $n_i = n_i'$ for $1 \leq i < j$ and $n_j < n_j'$. Then, let $\underline{m}$ be the maximal $k$-uple for this order and let $h_1, \dots , h_k$ denote a minimal set of generators associated with $\underline{m}$. 
We now assume that $\underline{m}$ is not independent, that is, there exists a $1 \leq r \leq k$ such that $m_r \in \Delta_{h_r}$. This means 
$$\sum_{i \neq r} a_ih_i = \tau^{m_r} + H.O.T.$$ for some $a_i \in \ring$. 
Then, we can consider the new generator $h_r' = h_r - \sum_{i \neq r} a_ih_i$ with $val_t(h_r') = m_r' > m_r = val_t(h_r)$. We still have that $ I = (h_1, \dots, h_r', \dots , h_k)$ and $$(m_1 , \dots, m_r' , \dots , m_k) > \underline{m}$$ which is absurd. 
Therefore, $\underline{m}$ is an independent $k$-uple for $I$ and the existence is proved.

Given the existence, we are now left with proving the uniqueness of the independent $k$-uple among all appearing from different representatives of $I$. 
Let $\underline{n}, \underline{n'}$ be two independent $k$-uples, say with $n_1 \leq n_1'$, let $I \in \KHilb{n}$ and let $g_1, \dots , g_k$ be a minimal set of generator such that $val_t(g_i) = n_i$. 
We want to prove that if we can find another minimal set of generators $g_1', \dots , g_k'$ of $I$ such that $val_t(g_i') = n_i'$, then $\underline{n} = \underline{n'}$.
By Theorem \ref{thm-matherequiv}, the change of representative is equivalent to the existence of an invertible matrix $$U 
   =\left( \begin{matrix}
    a_1^1 & \dots & a_1^k\\
    & \dots &\\ 
    a_k^1 & \dots & a_k^k
    \end{matrix} \right) \in GL_k(\ring)$$
such that
\begin{equation} \label{eq-mather}
    g_i = \sum_j a_j^i g_j'
\end{equation}
for every $i = 1, \dots , k$. We proceed by induction on $i$ to prove the equality $\underline{n} = \underline{n'}$.

From Equation (\ref{eq-mather}) for $i=1$, it immediately follows for degree reasons that $$n_1 = n_1' .$$ 

Then, we assume $n_i = n_i'$ for every $i = 1, \dots l$ and we can assume $n_{l+1} \leq n_{l+1}'$.
We consider Equation (\ref{eq-mather}) for $i = l+1$. We observe that for every $i = 1, \dots l$ by inductive hypothesis we have $n_i' = n_i < n_{l+1}$. Therefore, the equation implies
\begin{center}
    $n_{l+1} \in val_t(g_1, \dots , g_l) \subseteq \Delta_{g_{l+1}} \;\;$   or   $\; \; a^{l+1}_j = 0$ for every $j = 1, \dots l$.
\end{center}
Since the first condition would go against the independence hypothesis, it follows that $a^{l+1}_j = 0$ for every $j = 1, \dots , l$.
The same equation now becomes 
$$ g_{l+1}  = a_{l+1}^{l+1} g_{l+1}' + H.O.T.$$
and, since $n_{l+1}' \geq n_{l+1}$, for degree reasons it follows that $$n_{l+1}' = n_{l+1}.$$
\end{proof}

From now on, we will always assume to be working with independent $k$-uples.

\begin{rmk} \label{final-simplify-generators}
Applying this proposition to the set of generators of Remark \ref{principal-generators-reduced}, we obtain a presentation of the generators $f_i$ further simplified, since $\coprod_{j \neq i} n_j + \Gamma \subseteq \Delta_{f_i}$. This generating set is also compatible with one of Lemma \ref{def:ors-generators}. 
\end{rmk}

This allows us to define the following components.
\begin{df} \label{def:indep-k-uple}
Let $\underline{n}$ be an independent $k$-uple as in Definition \ref{def:independent-int-numbers}.
Following the notation of Definition \ref{general-notation}, we denote by $$\KHilb{n, \underline{n}} = 
\{ I \in \KHilb{n} \, | \, I = (g_1, \dots , g_k) , val_t(g_i) = n_i\} .$$
\end{df}

To conclude, we observe that thanks to Theorem \ref{uniqueness-intersection}, the components we just defined lead to the partition
$$\KHilb{n} = \coprod_{\underline{n}} \KHilb{n,\underline{n}}.$$
Many of these components will be empty as random $k$-uples $\underline{n}$ will not necessarily appear for a fixed length $n$. In the next section, we study what contributes to a given length $n$.

\subsection{Semigroup ideals appearing in $\KHilb{n,  \underline{n}}$} \label{section:length}
The goal of this section is to understand what $k$-uples contribute to a given length $n$. To do so, we require the additional information provided by the semigroup ideal structure of $\Gamma$, since we will later realize that the values appearing in the $k$-uple, that is, the $t$-adic valuation of the corresponding set of generators, are not enough. 
We start by introducing the following map.

\begin{df} \label{def-psi}
Let $(C,0)$ be a unibranch curve. 
Let $n,k \geq 1$ and let $\underline{n} = (n_1, \dots , n_k)$ be an independent $k$-uple as in Definition \ref{def:independent-int-numbers}. We define $$\Psi_{\underline{n}} : \prod_{i=1}^k \PHilb{n_i} \to \coprod_n\KHilb{n, \underline{n}} \,, \; ((f_1), \dots , (f_k)) \mapsto (f_1, \dots , f_k)$$
where we follow the notations of Remark \ref{principal-generators-reduced}.
\end{df}

We use principal Hilbert schemes $\PHilb{n_i}$ as their points exactly correspond to elements of valuation $n_i$.
Then, to determine the correct image of $\Psi_{\underline{n}}$, namely the length $n$, it is not enough to only consider the $k$-uple of valuations $\underline{n}$.

\begin{df} \label{x}
Let $\Delta \subseteq \Gamma$ be a semigroup ideal such that $|\Gamma \setminus \Delta| = n$. We denote by $$\KHilb{n, \underline{n}, \Delta} = \{ I \in \KHilb{n, \underline{n}} \, | \, val_t(I)=\Delta \} .$$
\end{df}

This is because, in general, given a semigroup ideal $\Delta$ in Definition \ref{x}, we only have the inclusion 
$$\coprod_i \{ n_i + \Gamma \} \subseteq \Delta$$ 
and not the equality.
This means that the set of non-negative integers $n_i$ might be just a proper subset of the generating set of the semigroup ideal $\Delta$, and the equality is obtained through non-trivial $\ring$-combinations of $f_i$. In other words, we need to take into account the syzygies of $\Delta$ to determine the correct length. This will be done in the next section, where we substitute the constraint $val_t(I)=\Delta$, determining the component $\KHilb{n, \underline{n}, \Delta}$ with the equation arising from the syzygies of $\Delta$.

Finally, using these new components we obtain the partition $$\KHilb{n} = \coprod_{\underline{n}, \Delta } \KHilb{n, \underline{n}, \Delta} $$
which was our initial goal.

We now introduce some morphisms that will support us in studying the map $\Psi_{\underline{n}}$ and explicating the relationship between one-generator and fixed-generators loci.

\begin{df} \label{proj}
Let $n \geq 1$ and let $\Delta \subseteq \Gamma$ be a semigroup ideal such that $n \not\in \Delta$.
Regarding $\PHilb{n}$ as an affine space, as per Remarks \ref{simplify-reps}, we define the morphism 
$$\pi_\Delta: \PHilb{n} \to \PHilb{n} \, , \; 
I = (\lambda_i)_{i \in \{ n+H \} \cap \Gamma} \mapsto  (\mu_i)_{i \in \{ n+H \} \cap \Gamma}$$
with $$
\mu_i =
\begin{cases}
  \lambda_i & \text{for } i \in \{ n+H \} \cap \Gamma \setminus \Delta \\
  0 & \text{otherwise.}
\end{cases}
$$
\end{df}

\begin{rmk} \label{rmk-change-reps}
We observe that the generators $\pi_\Delta(f_i)$ exactly correspond to the generators considered in Lemma \ref{def:ors-generators}.
\end{rmk}

The morphisms $\pi_\Delta$ allow us to express $\Psi_{\underline{n}}$ in a more informative way, making us able to prove the following result.

\begin{thm} \label{trivial-fib} \label{thm:multiple-fibers}
Let $\Delta \subseteq \Gamma$ be a semigroup ideal and let $I \in \KHilb{n, \underline{n}, \Delta}$.
Let $S_i = \{ n_i + H \} \displaystyle\cap \Delta $ and let
$F_{{i}} = \mathbb A^{|S_{{i}}|}$.
The map $\Psi_{\underline{n}}$ is a trivial fibration with fiber $$\prod_i F_{{i}} = \mathbb{A}^{\sum_i |S_i|} $$ that is, 
$$\Psi_{\underline{n}}^{-1}(\KHilb{n, \underline{n}, \Delta}) \cong \KHilb{n, \underline{n}, \Delta} \times \mathbb{A}^{\sum_i |S_i|}  .$$
\end{thm}

\begin{proof}
We recall that $\Psi_{\underline{n}}$ is defined as (see Remark \ref{rmk-change-reps})
$$ ( (f_1), \dots , (f_k) ) \mapsto (f_1, \dots, f_k) = ( \pi_\Delta((f_1)), \dots , \pi_\Delta((f_k)) )  .$$ 
We denote by $I_j = \{ n_j + H \} \displaystyle\cap \Delta $ and we construct the morphism 
$$ \KHilb{n, \underline{n}, \Delta} \times F_{1} \times \dots \times F_{k} 
\to 
\Psi_{\underline{n}}^{-1}(\KHilb{n, \underline{n}, 
\Delta}), $$
$$(f_1, \dots , f_k) , (\lambda^1_i)_{i \in I_1}, \dots , (\lambda^k_i)_{i \in I_k} 
\mapsto $$
$$\left( \pi_\Delta(f_1) + \displaystyle\sum_{i \in I_1} \lambda^1_i \tau_{i} \right), \dots , \left(\pi_\Delta(f_k) + \displaystyle\sum_{i \in I_k} \lambda^k_i \tau_{i} \right)$$
and the inverse (up to change of $\operatorname{U}_\Delta$-representatives, see Remark \ref{rmk-change-reps}) is
$$\Psi_{\underline{n}}^{-1}(\KHilb{n, \underline{n}, \Delta}) \to \KHilb{n, \underline{n}, \Delta} \times F_{1} \times \dots \times F_{k}, $$
$$(f_1) , \dots , (f_k) \mapsto (\pi_\Delta(f_1) , \dots , \pi_\Delta(f_k)), (\lambda^1_i)_{i \in I_1}, \dots , (\lambda^k_i)_{i \in I_k}  .$$ 
\end{proof}

We translate the isomorphism of Theorem \ref{trivial-fib} into an equality of motivic classes.

\begin{crl}
Let $(C,0)$ be a unibranch curve. We have 
$$ [\KHilb{n, \underline{n}, \Delta}] =  [\Psi_{\underline{n}}^{-1}(\KHilb{n, \underline{n}, \Delta})]\Li^{ - \sum_{i=1}^k|\{n_i+H \} \cap \Delta|} $$
for every $n \geq 1, \underline{n}, \Delta$.
\end{crl}

This way, we obtain a first description for the motivic class $[\KHilb{n}]$.
\begin{crl}
Let $(C,0)$ be a unibranch curve. We have 
$$[\KHilb{n}] = \sum_{\underline{n}, \Delta} [\KHilb{n , \underline{n}, \Delta}] = \sum_{\underline{n}, \Delta} [\Psi_{\underline{n}}^{-1}(\KHilb{n, \underline{n}, \Delta})]\Li^{ - \sum_{i=1}^k|\{n_i+H \} \cap \Delta|} $$
for every $n,k \geq 1$.
\end{crl}

We are left with studying more in detail $[\Psi_{\underline{n}}^{-1}(\KHilb{n, \underline{n}, \Delta})]$. To do this, as mentioned at the end of Section \ref{section:length}, we will use conditions arising from the syzygies of $\Delta$, which will substitute the general condition $val_t(I)=\Delta$ defining $[\KHilb{n , \underline{n}, \Delta}]$ and explicitly describe $\Psi_{\underline{n}}^{-1}(\KHilb{n, \underline{n}, \Delta})$.

\begin{rmk}
Theorem \ref{thm:syz-ors} provides us with equivalent, simpler conditions to describe the semigroup of an ideal. For a fixed choice $p \in Syz$ of the syzygy variables, these conditions are linear in the coefficients $\lambda$ of the generators of the points of $\PHilb{n_i}$, hence they cut out an arrangement of affine hyperplanes in the space $Gen$ of Definition \ref{notation}, whose class only depends on the rank of the matrix of their coefficients.
We point out that the valuations $n_1, \dots , n_k$ of a minimal set of generators of $I$ and the generators $\nu_1, \dots , \nu_l$ of the semigroup ideal $\Delta = val_t(I)$ are in general two different families: one always has $\{ n_1, \dots , n_k \} \subseteq \{ \nu_1, \dots , \nu_l \}$, and the inclusion can be strict, as in Example \ref{e} below, since $\ring$-linear combinations of the generators can produce valuations outside $\coprod_i \{ n_i + \Gamma \}$. The two families coincide for $(p,q)$-curves by Equation (\ref{eq-s}).     
\end{rmk}



Finally, applying Theorem \ref{thm:syz-ors} we obtain the following result.

\begin{crl} \label{crl:stratum-class}
Let $(C,0)$ be a unibranch curve. We have $$[\KHilb{n, \underline{n}, \Delta}] =
\Li^{\sum_{i=1}^l |\Gamma_{> \nu_i} \setminus \Delta |- \sum_{i=1}^k|\{n_i+H \} \cap \Delta|}
\sum_{a = 1 \dots m_\Delta} 
[S_a]\Li^{ - a }$$
for every $n \geq 1, \underline{n}, \Delta$ where 
$$ (M_\Delta(p))_{i(jk)} =  u^j_i \phi_{\sigma_i - \nu_j} \phi_{k} \, , \; S_a = \{ p \in Syz \, | \, rk(M_\Delta (p))=a \} $$ is the stratum along which the motivic class of the hyperplane arrangement in $Gen$ is constantly equal to $$\Li^{\sum_{i=1}^l |\Gamma_{> \nu_i} \setminus \Delta | - a}.$$ 
\end{crl}

\begin{proof}
Here $\nu_1, \dots , \nu_l$ are the generators of $\Delta$, the index $j$ runs over $1, \dots , l$ and $k \in \Gamma_{> \nu_j} \setminus \Delta$, while the index $i$ and the entries of $M_\Delta(p)$ are to be understood as in Remark \ref{rmk:linearity}.
We follow the notations of Definition \ref{notation}.
For any $p = (x_{is}^{\sigma_i})_{is} \in Syz$, from Theorem \ref{thm:syz-ors} we obtain the hyperplane equations 
$$F_p^i (\lambda) = \sum_j \left(
u^j_i \phi_{\sigma_i - \nu_j} \phi_{\nu_j}  +
\sum_{c(\Delta) > s > \sigma_i } x^{\sigma_i}_{is} \phi_{s - \nu_j} \phi_{\nu_j} 
+
\sum_{k \in \Gamma_{> \nu_j} \setminus \Delta} u^j_i \lambda_k^{\nu_j} \phi_{\sigma_i - \nu_j} \phi_{k} 
\right)$$
for $i = 1, \dots , m_\Delta$.
Then, we consider the $m_\Delta \times \dim(Gen)$-matrix of the 0-th degree coefficients of the equations $F^i_p(\lambda)$ determining the hyperplanes, that is,
$$
(M_\Delta(p))_{i(jk)} =  u^j_i \phi_{\sigma_i - \nu_j} \phi_{k}  .
$$
Finally, its rank is enough to determine the generic hyperplane arrangement: if $rk(M_\Delta(p))=a$ then $$\dim(F^1_p \cap \dots \cap F^{m_\Delta}_p) = \dim(Gen) - a = \sum_{i=1}^l |\Gamma_{> \nu_i} \setminus \Delta | - a.$$
This gives us a partition of $Syz$ in strata along which the motivic class of the hyperplane arrangement in $Gen$ is constantly equal to $$\Li^{\sum_{i=1}^l |\Gamma_{> \nu_i} \setminus \Delta | - a}$$ and we are done.
\end{proof}

\begin{rmk} \label{rmk:linearity}
Two comments on the matrix $M_\Delta(p)$ are in order. First, the products $\phi_{\sigma_i - \nu_j}\phi_k$ are elements of $\ring$, and not scalars: by the entry $(M_\Delta(p))_{i(jk)}$ we mean the coefficients of the expansion of $u^j_i \phi_{\sigma_i - \nu_j}\phi_k$ in the basis $\{ \phi_m \}_{m \in \Gamma}$ of Definition \ref{def-syzy}, in the degrees $m < c(\Delta)$. Indeed, each condition $F^i_p(\lambda) = o(t^{c(\Delta)})$ of Theorem \ref{thm:syz-ors} is a finite system of linear equations in the variables $\lambda \in Gen$, one for every degree $m \in \Gamma$ with $m < c(\Delta)$, and $M_\Delta(p)$ is the matrix of coefficients of this system: its row index $i$ therefore runs over the pairs given by a syzygy of $\Delta$ and a degree below the conductor. Being $\lambda$ the variable, we read the 0-th degree coefficient of the equation $F_p^i(\lambda)$ from the $(i, jk)$-entry of the matrix $M_\Delta$, that we use to study the hyperplane arrangement.
Second, these conditions are linear in $\lambda$ provided that the terms collected in $H.O.T.$ in Theorem \ref{thm:syz-ors} do not contribute terms of degree $\geq 2$ in the variables $\lambda$ in the degrees below the conductor. This is automatic for $(p,q)$-curves, as showed in \cite[Section 3.1]{ors}.
\end{rmk}

This leads to a complete description of $[\KHilb{n}]$, which extends the results of \cite[Section 3.1]{ors} since no formula is given for $[\KHilb{n}]$.
\begin{crl} \label{crl-formula-unibranch}
Let $(C,0)$ be a unibranch curve. We have $$[\KHilb{n}] 
= \sum_{\Delta} 
\Li^{\sum_{i=1}^l |\Gamma_{> \nu_i} \setminus \Delta |}   
\sum_{\underline{n}} 
\Li^{- \sum_{i=1}^k|\{n_i+H \} \cap \Delta|}
\sum_{a = 1 \dots m_\Delta} 
[S_a]\Li^{ - a }$$
for every $n,k \geq 1$. 
\end{crl}

We now present an example of application of Corollary \ref{crl-formula-unibranch}, to further develop the computations presented in 
\cite[Section 7]{os} and \cite[Proposition 10]{piontk}. The ideals cannot be read from \cite[Section 7]{os}, and the number of generators is not controlled in \cite[Section 7]{os}. Moreover, this example will help with the understanding of Corollary \ref{crl:stratum-class}.

\begin{ex} \label{e}
We consider the ring $$\ring= \C[[t^4, t^6+t^7]]$$ with $$\Gamma = \{ 0, 4, 6, 8, 10, 12, 13, 14, 16, \dots \} . $$ The holes are $H = 1,2,3,5,7,9,11,15$ and $\delta=|H|=8$.

We compute motivic classes and Euler characteristics of $\KHilb{4}$ and match the Euler characteristics with the ones computed in \cite[Section 7]{os}. We start with $n=4$ because it is the first ``interesting'' length, since $\KHilb{2}$ and $\KHilb{3}$ just correspond to the same schemes but for $\mathbb A^2$.

In the partition of $\KHilb{4}$, it is enough to consider the index $\underline{n}$ as the syzygies are very easy to control, as presented in \cite[Proposition 10]{piontk}. The only independent $k$-uples appearing for $n=4$ are the following:
\begin{itemize}
    \item $k=1$. We only have $n_1=4$ and this corresponds to $\PHilb{4}$. We have that $$[\PHilb{4}]= \Li^3$$ (as confirmed by the Alexander polynomial or by the expression of Remark \ref{simplify-reps}) and $\chi(\PHilb{4})=1$.

    \item $k=2$. We have the $2$-tuples  
    $n_1 =6, n_2 = 17$ and $n_1 = 8, n_2 = 13$. We have that
    $$[\Psi_{\underline{n}}(\PHilb{8}, \PHilb{13})] = \Li$$
    $$[\Psi_{\underline{n}}(\PHilb{6}, \PHilb{17})] = \Li^2 .$$
    In total, $$[\Hilb^4_{0,2}(C)] = \Li + \Li^2$$ and $\chi(\Hilb^4_{0,2}(C))=2$.

    \item $k=3$. We have the $3$-tuples 
    $n_1=8 , n_2=10, n_3=19$ and $n_1=10 , n_2=12, n_3=13$. We have that
    $$[\Psi_{\underline{n}}(\PHilb{8}, \PHilb{10}, \PHilb{19})] = \Li^2$$
    $$[\Psi(\PHilb{10}, \PHilb{12}, \PHilb{13})] = 1 .$$
    In total, $$[\Hilb^4_{0,3}(C)] = \Li^2 +1$$ and $\chi(\Hilb^4_{0,3}(C))=2$.
\end{itemize}

We now detail the computation of the strata above. We recall that the length of an ideal $I$ is $|\Gamma \setminus val_t(I)|$, and that $[\PHilb{m}] = \Li^{|\{ m+H \} \cap \Gamma|}$ by Remark \ref{simplify-reps}, so that by Theorem \ref{trivial-fib} and Corollary \ref{crl:stratum-class} the class of each stratum is $\Li$ to the power $\sum_i |\{ n_i + H \} \cap \Gamma| - \sum_i |\{ n_i + H \} \cap \Delta|$, the syzygy contribution being trivial here.
\begin{itemize}
    \item $k=1$, $n_1 = 4$. Here $\Delta = 4 + \Gamma$ and $\Gamma \setminus \Delta = \{ 0,6,13,19 \}$ has $4$ elements, as expected. Since $\{ 4+H \} \cap \Gamma = \{ 6,13,19 \}$, Remark \ref{simplify-reps} gives $$\PHilb{4} = \{ (\tau_4 + \lambda_2\tau_6 + \lambda_9\tau_{13} + \lambda_{15}\tau_{19}) \, | \, \lambda \in \C^3 \} \cong \mathbb A^3 .$$
    \item $k=2$, $\underline{n} = (8,13)$. Here $\Delta = \Gamma_{\underline{n}}$ and $\Gamma \setminus \Delta = \{ 0,4,6,10 \}$. We have $\{ 8+H \} \cap \Gamma = \{ 10,13,17,19,23 \}$ and $\{ 13+H \} \cap \Gamma = \{ 14,16,18,20,22,24,28 \}$, while $\{ 8+H \} \cap \Delta = \{ 13,17,19,23 \}$ and $\{ 13+H \} \cap \Delta = \{ 13+H \} \cap \Gamma$, so that $$[\Psi_{\underline{n}}(\PHilb{8}, \PHilb{13})] = \Li^{5+7}\Li^{-(4+7)} = \Li .$$
    \item $k=2$, $\underline{n} = (6,17)$. Here $\Gamma \setminus \Gamma_{\underline{n}} = \{ 0,4,8,13 \}$, $\{ 6+H \} \cap \Gamma = \{ 8,13,17,21 \}$ and $\{ 17+H \} \cap \Gamma$ has $8$ elements, while $\{ 6+H \} \cap \Gamma_{\underline{n}} = \{ 17,21 \}$ and $\{ 17+H \} \cap \Gamma_{\underline{n}} = \{ 17+H \} \cap \Gamma$, so that $$[\Psi_{\underline{n}}(\PHilb{6}, \PHilb{17})] = \Li^{4+8}\Li^{-(2+8)} = \Li^2 .$$
    \item $k=3$, $\underline{n} = (10,12,13)$. Here $\Gamma \setminus \Gamma_{\underline{n}} = \{ 0,4,6,8 \}$ and $\{ n_i + H \} \cap \Gamma = \{ n_i + H \} \cap \Gamma_{\underline{n}}$ for every $i$, these sets having $6,7$ and $7$ elements respectively, so that $$[\Psi_{\underline{n}}(\PHilb{10}, \PHilb{12}, \PHilb{13})] = \Li^{20}\Li^{-20} = 1 .$$
    \item $k=3$, $\underline{n} = (8,10,19)$. This is the only stratum of $\KHilb{4}$ in which the semigroup ideal of the ideals is strictly larger than the one generated by the $k$-uple. Indeed $\Gamma \setminus \Gamma_{\underline{n}} = \{ 0,4,6,13,17 \}$, so that $\Gamma_{\underline{n}}$ alone would give length $5$, while the non-monomial structure of $\ring$ produces the extra valuation $17$: for $f_1 = t^8$ and $f_2 = t^4(t^6+t^7) = t^{10}+t^{11}$ we have $$(t^6+t^7)f_2 - t^8f_1 = 2t^{17}+t^{18}$$ so that $17 \in val_t(I)$ and $$\Delta = val_t(I) = \Gamma_{\underline{n}} \coprod \{ 17 + \Gamma \} \, , \; \Gamma \setminus \Delta = \{ 0,4,6,13 \}$$ has the expected $4$ elements. Notice that $\Delta$ has the $l=4$ generators $8,10,17,19$, while the ideals of this stratum have $k=3$ generators. Using $\Delta$ we have $|\{ 8+H \} \cap \Gamma| = 5$, $|\{ 10+H \} \cap \Gamma| = 6$, $|\{ 19+H \} \cap \Gamma| = 8$ and $\{ 8+H \} \cap \Delta = \{ 10,17,19,23 \}$, $\{ 10+H \} \cap \Delta = \{ 12,17,19,21,25 \}$, $\{ 19+H \} \cap \Delta = \{ 19+H \} \cap \Gamma$, so that $$[\Psi_{\underline{n}}(\PHilb{8}, \PHilb{10}, \PHilb{19})] = \Li^{5+6+8}\Li^{-(4+5+8)} = \Li^2 .$$
\end{itemize}
This last stratum illustrates the phenomenon described before Corollary \ref{crl:stratum-class}: the number $l$ of generators of $\Delta$ and the number $k$ of generators of the ideals parameterized by the stratum need not coincide for a curve that is not a $(p,q)$-curve.

These computations match the ones in \cite[Section 7]{os}. Indeed, we recall the formula for the generating series of the Euler characteristics in the special case of $\ring= \C[[t^4, t^6+t^7]]$ of \cite[Equation (14)]{os}. We have
$$\displaystyle\sum_{l,k}q^{2l}(1-a^2)^{k-1} \chi(\KHilb{l}) = \displaystyle\sum_{i, \Delta, k} q^{2(i-(8 - |\mathbb N \setminus \Delta|)) = 2l} (1-a^2)^{k-1} \chi(U_{i + \Delta, k})$$ 
since on the components $U_{i + \Delta} \subseteq \Hilb_0^l(C)$ the number of generators might not be constant.
According to \cite[Table 1]{os}, these are the only $i, \Delta$ contributing to $n=l=4$, and the number of generators and Euler characteristics can be extracted from \cite[Section 7]{os}.

\begin{itemize}
    \item $i=4, \Delta = 0$: it corresponds to our $n_1=4$. We have $\chi(\PHilb{4})=1$.

    \item $i=6, \Delta = (0, 11)$: it corresponds to our $n_1 =6, n_2 = 17$. We have that $\chi(\operatorname{U}_\Delta) = 1$ from \cite[Corollary 31]{os}.
    \item $i=8, \Delta = (0,5)$: it corresponds to our $n_1 = 8, n_2 = 13$. We have $\chi(\operatorname{U}_\Delta) =1$ from \cite[Corollary 31]{os}.
    \item In total, we have $\chi(\Hilb_{0,2}^4(C))=2$.

    \item $i=8 , \Delta = (0,2,9,11)$: it corresponds to our $n_1=8 , n_2=10, n_3=19$. We have $\chi(\operatorname{U}_\Delta) = 1$ from \cite[Pg. 18]{os}.
    \item $i= 10 , \Delta = (0,2,3)$: it corresponds to our $n_1=10 , n_2=12, n_3=13$. We have $\chi(\operatorname{U}_\Delta) = 1$ from \cite[Corollary 31]{os}.
    \item In total, we have $\chi(\Hilb_{0,3}^4(C)) = 2$.
\end{itemize}

One could think that also $i=6, \Delta=(0,2)$ contributes to $n=l=4$. However, that is not the case since $(6,8)$ is not a semigroup ideal coming from any actual ideal. Hence, the lists above are complete.

\end{ex}

We conclude by observing that the polynomiality and positivity of $[\KHilb{n}]$ are determined by the same properties for $[S_a]$. We mention here that the variety $S_a$ is just a subvariety of the rank $a$-determinantal variety, as not all matrices will appear as $M_\Delta(p)$ for some $p \in Syz$.

\section{The case of $(p,q)$-curves} \label{sec:pq-curves}

This section is dedicated to the case of $(p,q)$-curves, since all the results of Section \ref{sec:motivic-classes} can be further simplified in this case. This is mainly because given $I \in \KHilb{n, \underline{n}}$ we can immediately say that 
\begin{equation} \label{eq-s}
val_t(I) = \coprod_i \{ n_i + \Gamma \}    
\end{equation}
since the $\ring$-linear combinations of the $f_i$ exactly correspond to $\Gamma$-combinations of the $n_i$.

This is proved in the following theorem, also implying that the length is fully determined by the independent $k$-uple of Proposition \ref{def:indep-k-uple}. We also remark here that this statement follows from \cite[Corollary 23]{os} using Equation (\ref{eq-s}).
\begin{thm} \label{crl:int-numbers-hilb}
Let $(C,0)$ be a $(p,q)$-curve and let $I \in \KHilb{d, \underline{n}}$. We denote by
$$\Gamma_{\underline{n}}  = \coprod_i \{ n_i + \Gamma \} .$$
Then, the length of $I$ is $$d = n_1 - \delta + | \mathbb N \setminus  \coprod_i \{ (n_i - n_1) + \Gamma \}| = |\Gamma \setminus \Gamma_{\underline{n}}|$$
where $\delta$ denotes the delta invariant of the curve $(C,0)$.
\end{thm}
\begin{proof}
Let $Y =\mathcal{O}_C / I$ and $X=\C[[t]] / (t^{n_1})$. Let $$K = \operatorname{ker}(\mathcal{O}_C / I \xrightarrow[]{f} \C[[t]] / (t^{n_1})) = \operatorname{ker}(Y \xrightarrow[]{f} X) \text{ and } Q = \operatorname{coker}(\mathcal{O}_C / I \xrightarrow[]{f} \C[[t]] / (t^{n_1})) = \operatorname{coker}(Y \xrightarrow[]{f} X)$$ 
where the map $f: Y \to X$ is induced to quotients by the inclusion $i : \mathcal{O}_C \hookrightarrow \C[[t]]$ noticing that $i(I) = (t^{n_1})$. 
Using the exact sequence of finite $\C$-modules $$ 0 \to K \to Y \to X \to Q  \to 0$$ 
we want to find $d = \operatorname{dim}_\C \mathcal{O}_C / I = \operatorname{dim}_\C Y$.
We have \begin{equation} \label{eq1}
d = \operatorname{dim}_\C X + \operatorname{dim}_\C K - \operatorname{dim}_\C Q = n_1 + \operatorname{dim}_\C K - \operatorname{dim}_\C Q  .
\end{equation}
Then, consider the commutative diagram with exact sequences as rows 
\begin{center}
\begin{tikzcd}
0 \arrow[r] & I \arrow[r, hook] \arrow[d, hook] & \mathcal{O}_C \arrow[r, two heads] \arrow[d, hook] & \mathcal{O}_C / I \arrow[r] \arrow[d, "f"] & 0 \\
0 \arrow[r] & (t^{n_1}) \arrow[r, hook]         & {\C[[t]]} \arrow[r, two heads]                     & {\C[[t]] / (t^{n_1})} \arrow[r]     & 0 .
\end{tikzcd}
\end{center}
By the snake lemma, we get the exact sequence of finite $\C$-modules
$$ 0 \to K \to (t^{n_1}) / I 
\to \; \C[[t]] / \mathcal{O}_C 
\to Q\;  
\to 0  .$$ 
We recall that $\delta = \operatorname{dim}_\C \C[[t]] / \mathcal{O}_C$. Therefore, from the exact sequence we have $$\operatorname{dim}_\C Q = \operatorname{dim}_\C K - \operatorname{dim}_\C (t^{n_1}) / I  + \operatorname{dim}_\C \C[[t]] / \mathcal{O}_C  = \operatorname{dim}_\C K - \operatorname{dim}_\C (t^{n_1}) / I  + \delta  .$$
Combining this into Equation (\ref{eq1}) we get
$$d = n_1 + \operatorname{dim}_\C K - (\operatorname{dim}_\C K - \operatorname{dim}_\C (t^{n_1}) / I + \delta) = n_1 + \operatorname{dim}_\C (t^{n_1}) / I - \delta .$$
Finally, we are just left with computing $\dim_\C (t^{n_1}) / I$.
Since $(C,0)$ is a $(p,q)$-curve, we have that $\ring = \C[[t^p,t^q]]$ is monomial, so each generator of $I$ can be written as $t^{n_i}u_i$ with $u_i \in \C[[t]]^\times$.
We recall that $I= (t^{n_1}u_1 , \dots , t^{n_k}u_k)$ with $n_1 \leq \dots \leq n_k$. Up to a $t^{n_1}$-shift, the quotient $(t^{n_1}) / I$ is isomorphic to 
$$(1) / (u_1, \dots , t^{n_k - n_1} u_k) = \C[[t]] / \left( \ring + \dots + t^{n_k - n_1} \cdot \ring \right) $$ as $\C$-modules. In terms of the basis given by the powers of $t$, the quotient corresponds to removing the powers with exponent in $\Gamma \displaystyle\cup \dots \displaystyle\cup \{ (n_k-n_1) + \Gamma \}$ from those with exponent in $\mathbb N$, so that $$\operatorname{dim}_\C (t^{n_1})/I = \left| \mathbb N \setminus \displaystyle\coprod_i \{ (n_i-n_1) + \Gamma \} \right| $$ which gives the first equality of the statement. 
Since $(C,0)$ is a $(p,q)$-curve, by Equation (\ref{eq-s}) no cancellation of leading terms occurs, so that the above union is disjoint and $val_t(I) = \Gamma_{\underline{n}}$. The equality then follows from $\operatorname{dim}_\C \ring / I = |\Gamma \setminus val_t(I)|$.
\end{proof}

\begin{rmk}
We observe that Theorem \ref{crl:int-numbers-hilb} in the principal case recovers the usual intersection number, or also \cite[Lemma 9]{os}. Indeed, $| \mathbb N \setminus\coprod_i \{ (n_i - n_1) + \Gamma \} | =  | \mathbb N \setminus \Gamma | = \delta$ thus $n=n_1 - \delta + \delta = n_1$.
\end{rmk}

As for Theorem \ref{trivial-fib} and its proof, the map $\Psi_{\underline{n}}$ remains a trivial fibration.
\begin{crl} \label{crl:multiple-fibers}
In the same setting of Theorem \ref{crl:int-numbers-hilb}, let
$S_i = \{ n_i + H \} \displaystyle\cap \Gamma_{\underline{n}} $ and let
$F_{{i}} = \mathbb A^{|S_{{i}}|}$.
The map $\Psi_{\underline{n}}$ is a trivial fibration with fiber $$\prod_i F_{{i}} = \mathbb{A}^{\sum_i |S_i|} $$ that is, 
$$\prod_i [\PHilb{{n_i}}] \cong \KHilb{n, \underline{n}} \times \mathbb{A}^{\sum_i |S_i|}  .$$
\end{crl}

Finally, in the case of $(p,q)$-curves, we obtain a stronger structural result for $[\KHilb{n}]$ since from Remark \ref{simplify-reps} we know that $[\PHilb{n_i}]$ is a power of $\Li$.
We denote by $$d(\underline{n}) = n_1 - \delta + \left| \mathbb N \setminus \displaystyle\coprod_i \{ (n_i-n_1) + \Gamma \} \right| = |\Gamma \setminus \Gamma_{\underline{n}}|$$ the length provided by Theorem \ref{crl:int-numbers-hilb} and by 
$M_n = \{ \underline{n} \text{ independent} \, | \, d(\underline{n}) = n \}$. 
\begin{crl} \label{expression-monomial} \label{crl-formula-monomial}
Let $(C,0)$ be a $(p,q)$-curve. We have 
$$[\KHilb{n}] = \sum_{\underline{n} \in M_n} [\KHilb{n, \underline{n}}] =$$

$$\sum_{\underline{n} \in M_n} \left( \prod_i [\PHilb{{n_i}}] \right) \Li^{- \sum_i |\{ n_i + H \} \cap \Gamma_{\underline{n}}|} = \sum_{\underline{n} \in M_n} \Li^{ \sum_i |\{ n_i + H \} \cap \Gamma| - \sum_i |\{ n_i + H \} \cap \Gamma_{\underline{n}}|}$$ for every $n,k \geq 1$.
\end{crl}

\begin{rmk}
 As a side note, we observe that $M_n$ also corresponds to the set of all possible sets of semigroup ideal generators, for semigroup ideals such that $$|\Gamma \setminus \displaystyle\coprod_i \{ n_i + \Gamma \}| =n .$$   
\end{rmk}

Since the semigroup is equivalent to the topology as explained in \cite[Theorem 21]{Brieskorn}, the polynomials $[\KHilb{n}]$ also admit a topological meaning for curve singularities.

\begin{crl} \label{struct}
Let $(C,0)$ be a $(p,q)$-curve. Then, $[\KHilb{n}]$ is a polynomial in $\Li$ with positive coefficients for every $n,k \geq 1$. 
\end{crl}

In Section \ref{sec:ex} we present several applications of Corollary \ref{expression-monomial}, where we explicitly compute the generating series
$$\sum_{l,m} s^{2l} \omega(\Hilb_{0,m}^l(C)) \prod_{j=1}^m (1+t^{2j-1}v^2) = (\star)$$
    of weight polynomials of fixed-generators Hilbert schemes appearing in \cite[Conjecture 2]{ors} and check its relationship to link invariants.

We conclude this section by noting that Corollaries \ref{crl-formula-monomial} and \ref{struct} generalize the affine cell decomposition proposed by \cite[Theorem 12]{piontk}, as well as the results of \cite[Theorem 2.11]{gorsky-catalan} for the cell dimension. The corollaries also extend the results of \cite[Section 3.2]{ors} since the polynomiality can be deduced from \cite[Proposition 6]{ors}, but no formula is directly given for $[\KHilb{n}]$. Moreover, the positivity of the coefficients and the topological invariance, as well as the relationship to principal Hilbert schemes, are new to the best of our knowledge.

\section{Contact loci and motivic classes of unibranch Hilbert schemes of curves} \label{sec:princ-surj}
\label{sec:contact-loci-principal}

This section aims to relate the contact locus $\Cont^n_C = \Cont^n$ of order $n$ of a given unibranch curve $(C,0)$, introduced in Definition \ref{def:contact-locus}, to the principal Hilbert schemes of the given curve.
Thinking of points of $\Lo$ as parameterizations of unibranch curves in $\mathbb A^2$, we naturally have a map
to principal Hilbert schemes, opposite to the map $Z : \widehat{\mathcal{O}}_{C,0} \to \coprod_k S^k \Lo$ introduced in \cite{sympow-gorsky}. 

\begin{rmk}
We note that none of the results in this section strictly utilize the unibranch condition, which could be removed. We maintain this condition for consistency with the settings used throughout the rest of this paper.
\end{rmk}

\begin{rmk} \label{tr1}
By definition, the points of $\Cont^n \subseteq \Lo$ are pairs of formal power series $p,q \in \C[[t]]$ such that $val_t(f(p,q))=n$. From Lemma \ref{lemma-intersection-num}, up to excluding the measure-zero set of degenerate pairs in $\Cont^n$ (see Remark \ref{rmk-degenerate-par}), this means the curve parameterized by $(p,q)$ has intersection number $n$ with the curve $C$. Moreover, we observe that for every $I = (g, f) \in \PHilb{n}$ we have $(x,y)^n \subseteq I$. 
Therefore, different arcs $\gamma, \gamma'$ such that $\pi_{n+1}(\gamma) = \pi_{n+1}(\gamma')$ correspond to the same point of $\PHilb{n}$.  
\end{rmk}

\begin{rmk}
Following the notation of Definition \ref{def-shape-of-param}, let $i \in I$ index the set of coefficients $a_{m_i} \neq 0$. Then, the degenerate condition translates to the existence of some $k > 1$ such that $gcd(n,m_i) = k$ for every $i \in I$. From the proof of Theorem \ref{thm-puiseux}, it also follows that Puiseux parameterizations are non-degenerate.
\end{rmk}

\begin{df} \label{def:principal-surj}
Let $n\geq 1$. We define the map 
$$ \Pi: \pi_{n+1}(\Cont^n)  \to \PHilb{n} , \; \gamma \mapsto \operatorname{Im}(\gamma) = \operatorname{Im}(\gamma) \times_{\mathbb A^2} C$$ where $\operatorname{Im}$ denotes the scheme-theoretic image.
\end{df}

\begin{thm} \label{thm-morph}
The map $\Pi$ of Definition \ref{def:principal-surj} is a morphism.
\end{thm}
\begin{proof}
The proof follows the same construction as \cite[Theorem 3.14]{ila-curv}. More in detail, we immediately observe that $\Pi$ is a well-defined map since $\operatorname{Im}(\gamma)$ is a finite length scheme on $C$.
Moreover, working locally the map $\Pi$ is defined as $\gamma \mapsto (\operatorname{ker}(\gamma) , f)$ where $f \in \C[[x,y]]$ locally defines $C$. 

Then, we have to prove that $\Pi$ is a morphism. 
It is not straightforward to deduce it directly from its definition, so we will use the representability of the more general Hilbert functor of Definition \ref{def-fun}, constructing a morphism that will correspond to the same map on the $\C$-points of $\pi_{n+1}(\Cont^n)$ and $\PHilb{n}$. 
We denote by $\mathcal{Z}$ the universal subscheme, corresponding under representability to $id_C$ and consider the subvariety 
\begin{center}
    $Z_{J} \subseteq C \times \pi_{n+1}(\Cont^n)$, $Z_{J }
 = \{ ( supp(\phi) \cap m^k , \phi ) \}$.
\end{center}
The projection $p$ is finite, flat, and surjective of degree $n$, having as fibers the $n$ points of the support of the corresponding parameterized subscheme of length $n$ at $m$.
Then, we consider by representability the morphism $\Pi$ such that $Z_{J }= (H(\Pi))(\mathcal{Z}) =  \mathcal{Z} \times_{\PHilb{n}}\pi_{n+1}(\Cont^n)$.
\begin{center}
\begin{tikzcd}
& & & \mathcal{Z} \arrow[d, hook]\\
Z_J \arrow[r, hook] \arrow[urrr, bend left] \arrow[rd] & {C \times \pi_{n+1}(\Cont^n)} \arrow[d, "p"] \arrow[rr, "id_C \times \widetilde{\Pi} "] & & C \times \PHilb{n} \arrow[d, "\pi_2"] \\
& {\pi_{n+1}(\Cont^n)} \arrow[rr, "\widetilde{\Pi}"] &  & \PHilb{n}                          
\end{tikzcd}
\end{center}
By definition of $Z_J$, the morphism $\widetilde{\Pi}$ is then defined as associating to an arc the corresponding scheme as a point of the principal Hilbert scheme $\PHilb{n}$, corresponding to $\Pi$. Finally, on $\C$-points it corresponds to the explicit $\Pi$ considered above by construction.
\end{proof}

\begin{df} \label{def:unibr-scheme}
We say that a point $I \in \PHilb{n}$ is unibranch when $I \in \operatorname{Im}(\Pi)$. This equivalently means the principal ideal $I \subseteq \ring$ admits a generator $\Bar{g} \in \ring = \C[[x,y]]/(f)$ such that $g \in \C[[x,y]]$ is irreducible.
We denote by $$\UHilb^n_{0,1}(C) = \operatorname{Im}(\Pi)$$ the subset of $\PHilb{n}$ determined by the image of $\Pi$.
\end{df}

\begin{ex}
    We consider the cusp $(C,0)$ of Example \ref{ex-cusp}. We have 
    $$\UHilb^7_{0,1}(C) =\{ (t^7+\lambda t^8) \, | \, \lambda \in \C^\times \} $$
    (see Example \ref{exxx} for further details). Therefore, $\UHilb^7_{0,1}(C)$ is an open, non-closed subset of $\PHilb{7} \cong \mathbb A^1$, which shows that in general $\UHilb^n_{0,1}(C)$ is not a closed subscheme of the one-generator locus $\PHilb{n}$. We do not claim that $\operatorname{Im}(\Pi)$ is open for every $n$ and every curve: what we use in the following is only that $\UHilb^n_{0,1}(C)$ is a constructible subset of $\PHilb{n}$, being the image of a morphism of varieties by Theorem \ref{thm-morph}, so that its motivic class is well defined.
\end{ex}

The rest of the section will be dedicated to studying the image and the fibers of $\Pi$. In particular, the interest in $\UHilb_{0,1}^n(C) = \operatorname{Im}(\Pi)$ in the one-generator locus will be motivated by the Igusa zeta function, as its motivic class $[\UHilb^n_{0,1}(C)]$ will be related to the coefficients of $Z_C(s)$. 
We will also prove the existence of a threshold that makes $\UHilb^n_{0,1}(C)$ always non-empty.
However, this is not all: in Conjecture \ref{cj} we will also conjecture the surjectivity of $\Pi$, or equivalently the equality $\UHilb^n_{0,1}(C) = \PHilb{n}$,  for $n$ big enough. The example below shows that such an equality cannot hold for every $n$.

\begin{ex}  \label{ex-contactcusp}
We consider the cusp $(C,0)$ of Example \ref{ex-cusp}, with $f = x^3-y^2$, $\Gamma = \{ 0,2,3,4, \dots \}$, $H = \{ 1 \}$ and $\delta = 1$. We claim that $$\Cont^5 = \emptyset .$$
Indeed, let $\gamma = (p,q) \in \Lo$ be an arc with $a = val_t(p) \geq 1$ and $b = val_t(q) \geq 1$. If $3a \neq 2b$, then $$val_t(f(\gamma)) = val_t(p^3-q^2) = \min (3a, 2b)$$ which is either a multiple of $3$ or an even number, hence never equal to $5$. If $3a = 2b$, then $a = 2c$ and $b = 3c$ for some $c \geq 1$, the leading terms cancel and $val_t(f(\gamma)) > 6c \geq 6 > 5$. Therefore no arc has contact order $5$ with $C$. This also agrees with Theorem \ref{denef-contact}, since $5$ cannot be written as $\sum_{j \in J}k_jN_j$ with $k_j \geq 1$ for any of the strata $E^o_J$ meeting $\rho^{-1}(0)$, the multiplicities being $(N_i)_{i=1,2,3} = (2,3,6)$ as in Example \ref{ex-cusp} together with the multiplicity $1$ of the strict transform.
On the other hand, $\PHilb{5} \neq \emptyset$: since $\{ 5+H \} \cap \Gamma = \{ 6 \}$, Remark \ref{simplify-reps} gives $$\PHilb{5} = \{ (t^5 + \lambda t^6) \, | \, \lambda \in \C \} \cong \mathbb A^1 .$$ This implies that $\Pi$ cannot be surjective for $n=5$, since $\operatorname{Im}(\Pi) = \emptyset \subsetneq \PHilb{5}$.
\end{ex}

\subsection{Holes of the Igusa zeta function and unibranch Hilbert scheme} \label{sec:t}
The goal of this section is to show that the irreducibility condition of Definition \ref{def:unibr-scheme} is not too strict, as there exists a threshold $N$ from which $\UHilb^n_{0,1}(C) \neq \emptyset$ for $n \geq N$, hence defining an interesting part of $\PHilb{n}$. Moreover, we will see that the threshold is entirely determined by (the topology of) the curve $C$.

From the definition of $\Pi$, it follows that $\Cont^n = \emptyset$ if and only if $\UHilb^n_{0,1} = \operatorname{Im}(\Pi) = \emptyset$. We already presented an example of this behaviour in Example \ref{ex-contactcusp}. Showing that there exists a threshold from which the contact loci become non-empty is the central argument leading to the following result.

\begin{prp} \label{threshold}
    Let $\rho: (Y, E) \rightarrow (\mathbb A^2, C)$ be an embedded resolution of the pair $(\mathbb A^2, C)$ as in Definition \ref{res}.
    Let $N_i$ denote the multiplicity of $E_i$ as in Definition \ref{mult} and let $N = \operatorname{max}_i N_i$. Then, $$\UHilb^n_{0,1}(C) \neq \emptyset$$ for every $n \geq N$.
\end{prp}
\begin{proof}
We want to show that $\Cont^n$ is non-empty for $n \geq N$, since this would imply that $\UHilb_{0,1}^n(C)$ is also non-empty. 
Let $i'$ be the index of the divisor corresponding to $N$. Equivalently, it corresponds to the last divisor obtained by the last blow-up of $\rho$.
By Theorem \ref{denef-contact}, we know that $\Cont^n \neq \emptyset$ when 
\begin{center}
$E_J^o \neq \emptyset \;$ and $\; \displaystyle\sum_{j \in J} k_j N_j =n$    
\end{center}
for some $J \subseteq I$ and $k_j >0$.
Since $\rho$ is the embedded resolution of a curve, it is known that we can only have three cases:
\begin{itemize}
        \item Consider $E_J = E_i$: then, we have contributions to degrees $kN_j$ for every $k>0$. In particular, we have $N$ of the divisor $E_{i'}$.
        \item Consider $E_J = E_i \cap E_j$ for $E_i \cap E_j \neq \emptyset$: then, we have contributions to degrees $h N_i + e N_j$ for every $k,e >0$.
        \item Consider $E_J = E_{i'} \cap \Tilde{C}$ (as the last divisor is the only one intersecting the strict transform): then, we have contributions to degrees $h N + e \times 1$ for every $k,e >0$. In particular, we have $N + e$ for every $e > 0$.
\end{itemize}
\end{proof}

\begin{rmk}
    Proposition \ref{threshold} and its proof are stated in the setting of unibranch curves, as that is the case of interest for this paper. However, the proof admits an immediate generalization to the case of multibranch curves.
    Indeed, given $C = \coprod_{l=1}^b C_l$ a multibranch curve, it is enough to consider all the contributions $E_J = E_{i'} \cap \Tilde{C_l}$ for every $l=1, \dots, b$.
\end{rmk}

We conclude this part with a conjecture, supported by computations in some examples that we omit for the sake of brevity. We expect $\UHilb_{0,1}^n(C)$ not only to be non-empty, but actually to retrieve the entire one-generator locus if $n$ is big enough.
\begin{conj} \label{cj}
Let $\rho: (Y, E) \rightarrow (\mathbb A^2, C)$ be an embedded resolution of the pair $(\mathbb A^2, C)$ as in Definition \ref{res}.
    We have $$\UHilb^n_{0,1}(C) = \PHilb{n}$$ for $n \gg 0$.    
\end{conj}

\subsection{The space of branches and unibranch Hilbert schemes} \label{section:smooth-reps}

The goal of this section is to use $\Pi$ and the smooth reparameterizations of Definition \ref{def:smooth-reps} to relate $[\UHilb^n_{0,1}(C)]$, hence conjecturally $[\PHilb{n}]$ and $[\KHilb{n}]$ in the case of $(p,q)$-curves thanks to Corollary \ref{crl-formula-monomial}, to the motivic measure $\mu(\Cont^n) = \frac{[\pi_{n+1}(\Cont^n)]}{\Li^{2(n+1)}}$ of the $n$-th contact locus of $C$. 
We start by introducing the action of smooth reparameterizations on (truncated) arcs, originally considered in \cite{luengo}, and understand how it relates to Hilbert schemes.

\begin{df}  \label{def-smoothaction}
We denote by $S_0^{n+1}$ the group of automorphisms of $\C[[t]] / (t^{n+1})$ induced by elements of valuation $1$ in $\C[[t]] / (t^{n+1})$, as per Definition \ref{def:smooth-reps} under the truncation map $\pi_{n+1}$.
Noticing that any $\alpha \in S_0^{n+1}$ induces an automorphism of the scheme $\operatorname{Spec}(\C[[t]]/(t^{n+1}))$, that we denote by $\operatorname{Spec}(\alpha)$, and that a truncated arc $\phi \in \pi_{n+1}(\Cont^n)$ is a morphism $\phi : \operatorname{Spec}(\C[[t]]/(t^{n+1})) \to \mathbb A^2$, we consider on $\Lo$ the action induced by $S_0^{n+1}$ and the composition, defined as 
    \begin{center}
        $\bullet : S_0^{n+1} \times \pi_{n+1}(\Cont^n) \to \pi_{n+1}(\Cont^n)$, $(\alpha, \phi) \mapsto \phi \bullet \alpha = \phi \circ \operatorname{Spec}(\alpha)$, that is, the reparameterization of $\phi$ given by $t \mapsto \alpha$. 
    \end{center}    
\end{df}

\begin{rmk} \label{def:space-of-branches}
This is equivalent to the $\operatorname{Aut}_\C$-action of \cite[Section 2]{sympow-gorsky} and \cite[Section 2]{luengo}, as well as the action of \cite[Definition 3.4]{ila-curv}. Moreover, the quotient $$\mathcal{B} = \Lo / S_0^{n+1}, \, \, p_{b} : \Lo \to \mathcal{B}$$ of $\Lo$ by the $S_0^{n+1}$-action gives the space of branches $\mathcal{B}$ in the sense of \cite[Section 2, Definition]{sympow-gorsky} and \cite[Section 2, pg. 7]{luengo}. 
Given $\gamma \in \Lo$ we denote by $$b_\gamma = [\gamma] = S_0^{n+1} \bullet \gamma$$ its corresponding class in $\mathcal{B}$. The motivic measure $\mu$ and the truncation maps $\pi_k$ also descend to the branch space, as described in \cite{luengo}.
\end{rmk}

\begin{rmk}   \label{rmk2}
We want to relate the motivic measure $\mu(\Cont^n) = \frac{[\pi_{n+1}(\Cont^n)]}{\Li^{2(n+1)}}$ to $[\pi_{n+1}(\mathcal{B}^n)]$, as the latter will correspond to $[\UHilb_{0,1}^n(C)]$. 
We observe that the $S_0^{n+1}$-action is not free. Therefore, directly proving that $p_b$ is a Zariski locally trivial fibration can be hard. We will follow instead the reasoning of \cite[Definitions and Proposition 2, pg. 66-67]{luengo}.
We start by simplifying $\Cont^n$ with the partition $$\Cont^n =  \coprod_k \Cont^{n,k} , \; \; \Cont^{n,k} = \Cont^n \cap \pi_k^{-1}(0) \setminus \Cont^n \cap \pi_{k+1}^{-1}(0)  .$$
Then, given $b_\gamma \in \Cont^{n,k}$ 
we have that 
\begin{equation} \label{indep}
\mu(S_0^{n+1} \bullet \gamma) = \mu(p_b^{-1}(b_\gamma)) = \Li^{-k-1}(\Li -1) 
\end{equation} and this measure does not depend on $\gamma$. 
We finally obtain
\begin{equation} \label{prp2}
[\pi_{n+1}(\Cont^{n,k})] = \Li^{-k-1}(\Li -1) \Li^{n+1} [\pi_{n+1}(\mathcal{B}^{n,k})]
\end{equation}
where $p_b(\Cont^{n,k}) = \mathcal{B}^{n,k} \subseteq \mathcal{B}$ denotes the image of $\Cont^{n,k}$ in the branch space.
\end{rmk}

We now want to understand the set $\mathcal{B}^{n,k}$ in the space of branches, and what points of $\UHilb^n_{0,1}(C)$ correspond to points of $\mathcal{B}^{n,k}$. They turn out to be determined by ``easier'' valuation problems to solve.

\begin{prp} \label{legame}
The points of $\mathcal{B}^{n,k}$ are in one-to-one correspondence with curves with intersection number $n$ with $C$ and Puiseux parameterization $(t^k,t^m + H.O.T.)$, where $m = m(n,k,C)$ depends only on $n,k$ and $C$.  
\end{prp}
\begin{proof}
Let $b_\gamma \in \mathcal{B}^{n,k}$ for some $\gamma \in \Cont^{n,k} \subseteq \Cont^n$. Then, by definition $ \gamma = (t^k, t^mu)$  with $m \geq k$ and $u \in 1+ t \C[[t]]$. 
Moreover, $\gamma \in \Cont^n$ additionally implies that $f(\gamma) = f(t^k, t^m u) = t^n v$ for some $v \in \C[[t]]^\times$. The linear equation
$$val_t(f(t^k, t^m u)) = val_t(t^n v) = n$$
determines $m$ in terms of $C,n$ and $k$. Indeed, writing $f = \sum_{(a,b)} c_{ab}x^ay^b$ we have $$f(t^k, t^mu) = \sum_{(a,b)} c_{ab}u^bt^{ak+bm}$$ so that, whenever the minimum $\min \{ ak+bm \, | \, c_{ab} \neq 0 \}$ is attained at a single vertex $(a_C,b_C)$ of the Newton polygon of $f$, the condition above becomes the linear equation $$a_Ck + b_Cm = n$$ with $a_C,b_C \geq 0$ determined by $f=f_C$ and by $k$, which determines $m$. When the minimum is attained along an edge of the Newton polygon, the leading terms cancel and $m$ is determined by the same argument applied to the resulting higher-order terms, as for the cusp of Example \ref{exxx}.
\end{proof}

\begin{rmk} \label{tr2}
We observe that all these curves are identified if they differ by a smooth parameterization, or in degree $\geq n$. This means that $\pi_r(t^m + H.O.T)$ determines the corresponding point of the Hilbert scheme for $$r = \left\lfloor \frac{n + 2\delta -1}{d_C} \right\rfloor +1$$ where $d_C = val_t(q_C)$ (we use the notation of Definition \ref{def:puiseux-param}). This is because any coefficient $u_i$ of $t^mu$ for $i> \left\lfloor \frac{n + 2\delta -1}{d_C} \right\rfloor$ appears in terms of valuation $> n + 2\delta -1$ in $\ring$ hence not effectively contributing to $\PHilb{n}$. This is the optimal bound.
\end{rmk}

We immediately see that the truncations of Remarks \ref{tr1} and \ref{tr2}, together with the action introduced in Definition \ref{def-smoothaction}, completely determine the fiber of $\Pi$. 

\begin{prp} \label{fibers-pi}
Let $I \in \UHilb^n_{0,1}(C)$. Let $g \in \ring$ be any generator of $I$ and let $\gamma_g \in \pi_{n+1}(\Cont^{n,k})$ denote the Puiseux parameterization of the corresponding curve $V(g) \subseteq \mathbb{A}^2$ for some $k$.
Then, we have  $$[\Pi^{-1}(I)] = (\Li -1) \Li^{-k-1} [F^{n,k}]$$ 
where $F^{n,k}$ denotes the fiber of the truncation map $$\pi_r : \pi_{n+1}(p_2(\mathcal{B}^{n,k})) \to \pi_r(p_2(\mathcal{B}^{n,k}))$$ with $r$ as in Remark \ref{tr2}.
\end{prp}
\begin{proof}
We already observed in Remark \ref{tr2} that if $$\pi_r(p_2(\gamma_g)) = \pi_r (p_2(\gamma))$$ then $$(g) = (g_\gamma)$$ as points of $\PHilb{n}$.
Then, by Lemma \ref{back:intersection-numb} we have that $val_t(C(\gamma_g)) = n$ and by definition $$\Pi^{-1}(I) = \{ \gamma \in \pi_{n+1}(\Cont^{n,k}) \, | \, \operatorname{Im}(\gamma) = V(g) \} = $$
$$\{  \gamma_g \circ \alpha \in \pi_{n+1}(\Cont^{n,k}) \, | \, \alpha \in \operatorname{Aut}^{n+1}_{\C ,0} \} = S_0^{n+1} \bullet \gamma_g$$ since 
all equivalent parameterizations of the same curve are described in Definition \ref{def:smooth-reps}.
Finally, using the $\mathcal{C}$-equivalence presented in Definition \ref{def-matherequiv} in $\Cont^{n,k}$, it is immediate to see that in the principal case any other representative of $I$ is of the form $$g_1 = u g$$ with $u \in \ring^\times$. Therefore, we have $\gamma_g = \gamma_{g_1}$ 
and, using Equation (\ref{indep}), we are done.    
\end{proof}  

\begin{rmk} \label{rmk:explicit-Fnk}
The class $[F^{n,k}]$ can be computed explicitly. By Proposition \ref{legame}, a point of $\pi_{n+1}(p_2(\mathcal{B}^{n,k}))$ is a truncated series $$t^m + \displaystyle\sum_{m < j \leq n} u_jt^j$$ with $m = m(n,k,C)$, and the requirement $val_t(f(t^k,t^mu)) = n$, rather than $> n$, forces a first non-vanishing condition on its coefficients: there exists a critical degree $d = d(n,k,C) \leq n$, determined by the Newton polygon of $f$ and by the pair $(k,m)$ as in the proof of Proposition \ref{legame}, such that $u_j = 0$ for $m < j < d$ and $u_d \neq 0$, while the coefficients $u_j$ with $j > d$ are free. Since by Remark \ref{tr2} two such jets determine the same point of $\PHilb{n}$ if and only if their $\pi_r$-truncations agree, the fiber $F^{n,k}$ is cut out by the coefficients of degree $\geq r$ and we obtain $$[F^{n,k}] = \begin{cases} \Li^{n+1-r} & \text{ if } d < r \\ (\Li-1)\Li^{n-d} & \text{ if } d \geq r \end{cases}$$ the factor $(\Li-1)$ accounting for the condition $u_d \neq 0$ when the critical degree is not already fixed by the truncation. All the fibers computed in Section \ref{sec:ex} are recovered in this way.
\end{rmk}

\begin{rmk} \label{leg}
    When changing representative to $g_1 = uf + vg$ for $u \in \C[[x,y]]^\times$, $v \in \C[[x,y]]$, we leave $\Cont^{n,k}$ for $\Cont^{n, mult_C}$. While all jets are needed to compute $[\pi_{n+1}(\Cont^n)]$ and we consider both $\gamma_g$ and $\gamma_{g_1}$ in the computation, when computing $[\UHilb^n_{0,1}]$ we can use just one of the representatives, so $[\pi_{n+1}(\Cont^{n,k})]$ for example is enough.
\end{rmk}

\begin{rmk} \label{map-descend-branch}
With the previous proof, we equivalently observed that $\Pi$ is $S_0^{n+1}$-invariant. This leads to the map  
$$
\Pi_b : \pi_{n+1}(\mathcal{B}^n) \to \UHilb^n_{0,1}(C), \, \, b_\gamma = p_b(\gamma) \mapsto \Pi(\gamma)
$$
where $\mathcal{B}^n = p_b(\Cont^n)$ denotes the quotient of $\Cont^n$ by the $S_0^{n+1}$-action, as introduced in Remark \ref{def:space-of-branches}, and $p_b$ commutes with $\pi_{n+1}$ as explained in \cite{luengo}.    
\end{rmk}

Knowing the fibers of the morphism $\Pi$, hence of $\Pi_b$, we are now finally able to use it to relate the motivic class of $\UHilb_{0,1}^n(C)$ to the $n$-th contact locus $\Cont^n$, thanks to Proposition \ref{fib}.

\begin{prp} \label{prp1} \label{prp-def-Fnk}

    The restriction $\Pi_b^k: \pi_{n+1}(\mathcal{B}^{n,k})\to \UHilb^n_{0,1}(C)$ of the map $\Pi_b$ defined in Remark \ref{map-descend-branch} is surjective, and all its fibers are isomorphic to the variety $F^{n,k}$ described in Proposition \ref{fibers-pi}. {Moreover, it is a fibration which is locally trivial for the Zariski topology, being the truncation map $\pi_r$ of Proposition \ref{fibers-pi} up to the $S_0^{n+1}$-action.} Proposition \ref{fib} then implies $$[\operatorname{Im}(\Pi_b^k)] [F^{n,k}] = [\pi_{n+1}(\mathcal{B}^{n,k})]$$ for every $n,k \geq 1$.
\end{prp}

\begin{proof}
We reiterate the steps of \cite[Theorem 3.14]{ila-curv}. We already proved in Theorem \ref{thm-morph} that $\Pi$ is a morphism, and with Proposition \ref{fibers-pi} we have the description of the fiber of $\Pi_b^k$. Finally, it is surjective by definition since $\UHilb^n_{0,1}(C) = \operatorname{Im}(\Pi)$.
\end{proof}

We finally have all the information needed to relate contact loci with the unibranch locus. Indeed, putting together Proposition \ref{prp1} and Remark \ref{rmk2}, we obtain the following chain of relationships. 
\begin{crl} \label{crrr1}
Let $(C,0)$ be a unibranch curve. Let $\operatorname{Im}(\Pi_b^k) \subseteq \UHilb^n_{0,1}(C)$ denote the image of the restriction $\Pi_b^k: \pi_{n+1}(\mathcal{B}^{n,k}) \to \UHilb^n_{0,1}(C)$ of the map $\Pi_b$ introduced in Remark \ref{map-descend-branch}. We have
$$ [\pi_{n+1}(\Cont^{n,k})] = \Li^{n-k}(\Li -1) [\operatorname{Im}(\Pi_b^k)][F^{n,k}]$$
for every $n,k \geq 1$.
\end{crl}

This also means we just obtained the coefficients of the Igusa zeta function 
$$Z_C(s)=\int_{\pi_1^{-1}(0)} (\Li^{-ord_C})^s d\mu = \sum_{n \geq 1} \mu(\Cont^n_C) \Li^{-ns}  = \sum_{n \geq 1} \frac{[\pi_{n+1}(\Cont^{n})]}{ \Li^{2(n+1)}} \Li^{-ns}$$ from the unibranch locus in the one-generator locus of the Hilbert scheme of the given curve.
\begin{crl} \label{form}
Let $(C,0)$ be a unibranch curve and let $Z_C(s)$ 
be the Igusa zeta function of $C$. Then, we have
$$Z_C(s)= (\Li-1) \Li^{-2}\sum_{n \geq 1}\Li^{-n(s+1)} \sum_k \Li^{-k}[\operatorname{Im}(\Pi_b^k)] [F^{n,k}]$$  
and $[\operatorname{Im}(\Pi_b^k)]$ is related to $[\UHilb_{0,1}^n(C)]$ in Proposition \ref{legame}.
\end{crl}
\begin{rmk}
We mention here that the relationship between principal Hilbert schemes and motivic integration was also explored in \cite{wyss}. However, we still do not know the connection between our and that work, and our connection between $\UHilb_{0,1}^n(C)$ and $\PHilb{n}$ is just conjectural.
\end{rmk}

Finally, combining Theorem \ref{denef-contact} with \cite[Proposition 3.24] {ila-curv} it follows that $[\pi_{n+1}(\Cont^{n,k})]$ is fully determined by an embedded resolution of $(C,0)$. As a consequence, we obtain an explicit expression in terms of certain divisors for $[\operatorname{Im}(\Pi_b^k)]$ and $[\UHilb_{0,1}^n(C)]$. 

\begin{crl} \label{c}
Let $(C,0)$ be a unibranch curve.
Let $\rho: (Y, E) \rightarrow (\mathbb A^2, C)$ be an embedded resolution as in Definition \ref{res}. 
Then, we have
$$[\operatorname{Im}(\Pi_b^k)] = \frac{1}{(\Li-1)\Li^{n}} \frac{\Li^k}{[F^{n,k}]}
\sum_{\substack{J \subseteq I, J \neq \emptyset, j \in J \\ k_j \geq 1 \, | \, k =\sum m_j k_j}}
(\Li-1)^{|J|-1}[E_J^o]
\left( \sum_{j \in J, \, \sum k_jN_j = n}
\Li^{-\sum_{j \in J} k_j \nu_j} \right)$$
for every $n,k \geq 1$, where $m_j$ is defined in \cite[Definition 3.21]{ila-curv}, $E_J^o$ is the open stratum of Definition \ref{def:strata} associated to $J \subseteq I$, $N_j$ and $\nu_j$ are respectively the multiplicity and the discrepancy of $E_j$ of Definition \ref{mult}, and $[F^{n,k}]$ is described in Proposition \ref{fibers-pi} and Remark \ref{rmk:explicit-Fnk}. The class $[\operatorname{Im}(\Pi_b^k)]$ is related to $[\UHilb_{0,1}^n(C)]$ by Proposition \ref{prp-def-Fnk} and Remark \ref{leg}. 
\end{crl}

Using this expression, we can deduce structural and invariance properties for $[\UHilb^n_{0,1}(C)]$.

\begin{crl}
Let $(C,0)$ be a unibranch curve. Then, $[\UHilb^n_{0,1}(C)]$ is a polynomial in $\Li$ for every $n \geq 1$. This polynomial is also a topological invariant of $(C,0)$ since the resolution is.
\end{crl}

\begin{proof}
From Remark \ref{leg} and Proposition \ref{legame}, we can prove the polynomiality in $\Li$ of $[\UHilb_{0,1}^n(C)]$ by proving it for $[\operatorname{Im}(\Pi_b^k)]$. To prove the polynomiality of $[\operatorname{Im}(\Pi_b^k)]$, we prove that all the motivic classes involved in Corollary \ref{c} are polynomial in $\Li$.
From Proposition \ref{fibers-pi} and Remark \ref{rmk:explicit-Fnk}, it follows that $[F^{n,k}]$ is a power of $\Li$, up to a factor $(\Li-1)$.
Then, since divisors can only be $\mathbb P^1$, $E_J^o$ is a $\mathbb P^1$ with a finite number $k_J$ of points removed, or a point. Therefore, we have $$[E_J^o] = (\Li +1) - k_i \; \text{ or } \; [E_J^o] = 1 .$$
The equivalence between the geometric and topological data then follows from \cite[Theorem 21]{Brieskorn}.
\end{proof}

From these corollaries, Conjecture \ref{cj} would lead to an expression for $[\PHilb{n}]$, and in the case of $(p,q)$-curves for $[\KHilb{n}]$, in terms of an embedded resolution. This would be in perfect agreement with the topological invariance obtained in Remark \ref{simplify-reps} and Corollary \ref{struct}, respectively.

\section{Examples} \label{sec:ex}
We now provide some examples, in particular as applications of Corollary \ref{form}, where we explicitly compute $[F^{n,k}]$ and the contact loci of some given curves in relation to their unibranch locus in the one-generator locus.

\begin{ex}
Let $C$ be the line, locally defined by the polynomial $f=x \in \C[[x,y]]$. Hence, we have $\ring = \C[[t]]$.
The unibranch locus is 
\begin{center}
$\UHilb^n_{0,1}(C) = \PHilb{n} = \{ (t^n) \}$\\
    \vspace{1mm}
$[\UHilb^n_{0,1}(C)]=[\operatorname{Im}(\Pi_b^1)] = 1$
\end{center}
and $$[F^{n,1}] = \Li^n$$
since from Proposition \ref{legame} and Remark \ref{tr2} we obtain $r=1$ and $\gamma = (t^n, 0)$ is the only jet to consider. 
Therefore, up to smooth parameterization, we obtain 
$$ \mu(\Cont^n) = (\Li - 1)\Li^{-n-1}1\Li^{-1} = (\Li^{-n} - \Li^{-n-1}) \Li^{-1} $$ 
and this matches the measure of $\Cont^n = (t^n) \setminus (t^{n+1}) \times (t)$ that can be computed using well-known methods.
\end{ex}

\begin{ex} \label{exxx}
Let $C$ be the cusp of Example \ref{ex-cusp}, locally defined by the polynomial $f=x^3-y^2 \in \C[[x,y]]$ and $\ring = \C[[t^2, t^3]]$. We start with some concrete examples and then provide a general formula for $\Cont^n$ that matches the one provided in \cite[Example 3.1.6]{greenbook}.

\begin{itemize}
    \item $n=4$. The only stratum is $k=2$, $m \geq 2$. 
    
    We have 
    \begin{center}
    $\UHilb^4_{0,1}(C) = \PHilb{4} = \{ (t^4 + \lambda t^5) \, | \, \lambda \in \C \}$\\
    \vspace{1mm}
    $[\UHilb^4_{0,1}(C)] = [\operatorname{Im}(\Pi_b^2)] = \Li$
    \end{center}
    and $$[F^{n,2}] = \Li^2$$
    since from Proposition \ref{legame} and Remark \ref{tr2} we obtain 
    $r=3$ 
    and $\gamma = (u_0t^2 + u_1 t^3 + u_2t^4, t^2)$ for $u_0 \in \C$ is the generic jet to consider, corresponding to the ideal $(t^4 + u_0 t^5)$.
    Therefore, up to smooth parameterization, we obtain 
    $$ \mu (\Cont^4) = (\Li - 1)\Li^{-2-1}\cdot\Li^{-2} = \Li^{-4} - \Li^{-5} .$$  
\end{itemize}

\begin{itemize}
    \item $n=7$. The only stratum is $k=2$, $m = 3$.
    
    We have
    \begin{center}
    $\UHilb^7_{0,1}(C) =\{ (t^7+\lambda t^8) \, | \, \lambda \in \C^\times \}$\\
    \vspace{1mm}
    $[\UHilb^7_{0,1}(C)]=[\operatorname{Im}(\Pi_b^2)] = \Li - 1$ 
    \end{center}
    and $$[F^{n,2}] = \Li^3$$
    since 
    $r=5$ 
    and $\gamma = (t^2, t^3 + u_1 t^4 + \dots + u_4t^7)$ for $u_1 \in \C^\times$ is the generic jet to consider, corresponding to the ideals $(t^7 + u_1 t^8)$.
    Therefore, up to smooth parameterization, we obtain 
    $$\mu(\Cont^7) = (\Li - 1) \Li^{-2-1} \cdot (\Li - 1) \cdot \Li^{-5} = 
    (\Li -1) (\Li^{-7} - \Li^{-8}) .$$
\end{itemize}

    \begin{itemize}
        \item $n=13$. We work with the strata $k=2$, $m = 3$ and $k=4$, $m = 6$. 
        
        We have
        \begin{center}
        $\UHilb^{13}_{0,1}(C) = \PHilb{13} = \{ (t^{13}+\lambda t^{14}) \, | \, \lambda \in \C^\times \} \cup \{ (t^{13})  \}$\\
        \vspace{1mm}
        $[\UHilb^{13}_{0,1}(C)]=[\operatorname{Im}(\Pi_b^4)] + [\operatorname{Im}(\Pi_b^2)]= (\Li -1) + 1 = \Li$
        \end{center}
        and 
        \begin{center}
        $[F^{n,4}] = \Li^6$\\
        \vspace{1mm}
        $[F^{n,2}] = (\Li-1)\Li^3$    
        \end{center}
since we obtain $r=8$ 
and $\gamma_4 = (t^4 , t^6 + u_1 t^7 + \dots + u_7 t^{13})$ for $u_1 \in \C^\times$, $\gamma_2 = (t^2, t^3 + u_1t^{10} + \dots + u_4 t^{13})$ for $u_1 \in \C^\times$ are the generic jets to consider, corresponding to the ideals $(t^{13} + u_1 t^{14})$, $(t^{13} + H.O.T) = (t^{13})$ respectively.
    Therefore, up to smooth parameterization, we obtain 
    $$\mu(\Cont^{13,2}) = (\Li - 1) \Li^{-2-1} \cdot (\Li - 1) \Li^{-11} = 
    (\Li -1) (\Li^{-13} - \Li^{-14}) $$
    and $$ \mu(\Cont^{13, 4}) = 
    (\Li - 1) \Li^{-4-1} \cdot (\Li - 1)  \cdot \Li^{-8} = (\Li - 1) (\Li^{-12} - \Li^{-13}).
    $$
    Finally, we obtain $$\mu (\Cont^{13}) = \Li^{-11} -\Li^{-12} -\Li^{-13} + \Li^{-14}  .$$ 
\end{itemize}
Compared to the computations of \cite[Example 3.1.6]{greenbook}, this approach is simpler as we deal with easier subsets of contact loci whose points are further simplified thanks to the assumptions on their coefficients: using smooth parameterizations, the first jet of the pair is always just a $t$-power, and the coefficients of the second jet of the pair are zero or one.

We conclude with the general formulas for $\mu(\Cont^n)$. We explicitly write the details just for the cases $n=6k, 6k+1$, as the remaining cases are identical.
    \begin{itemize}
        \item $n=6k$. We work with the strata $\Cont^{6k,2i}$ for $i=1, \dots ,k$ and we have
        \begin{center}
        $\UHilb^{6k}_{0,1}(C) = \PHilb{6k} = \{ (t^{6k}+\lambda t^{6k+1}) \, | \, \lambda \in \C \}$\\
        \vspace{1mm}
        $[\UHilb^{6k}_{0,1}(C)] = [\operatorname{Im}_b^{2k}]= \Li$.
        \end{center}

        For $i=1, \dots , k-1$ the stratum is determined by the generic jets $\gamma_i = (t^{2i}, t^{3i} + u_1 t^{6k-3i} + \dots + u_{3i+1}t^{6k})$ for $u_1 \in \C^\times$. Since $r=3k+1$, have that $\gamma_i = (t^{2i}, t^{3i})$ always corresponds to the ideal $I=(t^{6k} +  H.O.T.) = (t^{6k})$.
        Therefore, 
        $$[F^{n,2i}] = (\Li - 1) \Li^{
        (6k - (6k-3i) )
        -  (6k+1)} = \Li^{
        - 6k + 3i-1}$$
        and
        $$\mu(\Cont^{6k,2i})  = (\Li-1)\Li^{-2i-1} \cdot (\Li - 1) \Li^{
        - 6k + 3i-1} .$$

        For $i=k$ we have that $\gamma_k = (t^{2k}, u_1 t^{3k})$ for $u_1 \in \C$ corresponds to the ideals $I= (t^{6k} + u_1 t^{6k+1})$.
        Therefore, 
        $$[F^{n,2k}] = \Li^{
        (6k - (6k-3k) )
        -  (6k+1)} = \Li^{
        - 3k-1)}$$
        and $$\mu(\Cont^{6k,2k}) = (\Li-1)\Li^{-2k-1} \cdot \Li \cdot \Li^{-3k -1} = (\Li-1)\Li^{-2k-1}  \cdot \Li^{
        - 3k } .$$

        Finally, we obtain $$\mu(\Cont^n) = \sum_{i=1}^{k-1} (\Li-1)\Li^{-6k +3i -1}  (\Li^{-2i -1} - \Li^{-2i-1})$$
        $$ + (\Li^{-2k} - \Li^{-2k-1})\Li^{-3k} .$$

        \item $n= 6k+1$.
        We work with the strata $\Cont^{6k,2i}$ for $i=1, \dots ,k$ and we have
        \begin{center}
        $\UHilb^{6k+1}_{0,1}(C) = \PHilb{6k+1} = \{ (t^{6k+1}+\lambda t^{6k+2}) \, | \, \lambda \in \C \}$\\
        \vspace{1mm}
        $[\UHilb^{6k+1}_{0,1}(C)] = [\operatorname{Im}_b^{2k}] + [\operatorname{Im}_b^{2i}]= (\Li - 1) +1 = \Li$ for any $1 \leq i \leq k-1$.
        \end{center}

        For $i=1, \dots , k-1$ the stratum is determined by the generic jets $\gamma_i = (t^{2i}, t^{3i} + u_1 t^{6k-3i} + \dots + u_{3i+1)}t^{6k})$ for $u_1 \in \C^\times$. Since $r=3k+2$, have that $\gamma_i = (t^{2i}, t^{3i})$ always corresponds to the ideal $I=(t^{6k} +  H.O.T.) = (t^{6k})$.
        Therefore, 
        $$[F^{n,2i}] = (\Li - 1) \Li^{
        (6k - (6k-3i) )
        -  (6k+1)} = \Li^{
        - 6k + 3i-1}$$
        and
        $$\mu(\Cont^{6k,2i})  = (\Li-1)\Li^{-2i-1} \cdot (\Li - 1) \Li^{
        - 6k + 3i-1} .$$

        For $i=k$ we have that $\gamma_i = (t^{2k}, t^{3k} + u_1 t^{3k+1})$ for $u_1 \in \C^\times$ corresponds to the ideals $I= (t^{6k} + u_1 t^{6k+1})$.
        Therefore, 
        $$[F^{n,2k}] = \Li^{
        (6k - (6k-3k) )
        -  (6k+1)} = \Li^{
        - 3k-1}$$
        and $$\mu(\Cont^{6k,k}) = (\Li-1)\Li^{-2k-1} \cdot (\Li -1) \cdot \Li^{-3k -1} .$$

        Finally, we obtain $$\mu(\Cont^n) = \sum_{i=1}^{k} (\Li-1)\Li^{-6k + 3i-1}  (\Li^{-2i} - \Li^{-2i-1}) .$$

        \item $n= 6k+2$. We obtain $$\mu(\Cont^n) = \sum_{i=1}^{k} (\Li-1)\Li^{-6k + 3i-1}  (\Li^{-2i} - \Li^{-2j-1})$$ $$+ (\Li^{-3k-1} - \Li^{-3k-1-1})\Li^{-2k-1} .$$

        \item $n= 6k+3$. We obtain $$\mu(\Cont^n) = \sum_{i=1}^{k} (\Li-1)\Li^{-6k + 3i-1}  (\Li^{-2i} - \Li^{-2i-1})$$ $$+ (\Li^{-2k-1} - \Li^{-2k-1-1})\Li^{-3k-2} .$$

        \item $n= 6k+4$. We obtain $$\mu(\Cont^n) = \sum_{i=1}^{k} (\Li-1)\Li^{-6k + 3i-1}  (\Li^{-2i } - \Li^{-2j-1})$$ $$+ (\Li^{-3k-2} - \Li^{-3k-2-1})\Li^{-2k-2} . $$

        \item $n= 6k+5$. We obtain $$\mu(\Cont^n) = \sum_{i=1}^{k} (\Li-1)\Li^{-6k + 3i-1}  (\Li^{-2i } - \Li^{-2i-1})$$ $$+ (\Li^{-3k-1} - \Li^{-3k-1-1})\Li^{-2k-1} .$$
    \end{itemize}

In addition to being simpler to compute, we observe that the formulas just presented directly hold for the contact locus $\Cont^n$ and consequently the coefficients of $Z_C(s)$, instead of $\Cont^{\geq n}$ as for the ones presented in \cite[Example 3.1.6]{greenbook}.
\end{ex}

As applications of Corollary \ref{expression-monomial}, we now provide some examples where we explicitly compute the generating series
$$\sum_{l,m} s^{2l} \omega(\Hilb_{0,m}^l(C)) \prod_{j=1}^m (1+t^{2j-1}v^2) = (\star)$$
of weight polynomials of fixed-generators Hilbert schemes appearing in \cite[Conjecture 2]{ors} for some given $(p,q)$-curves, and check their relationship to link invariants.

\begin{ex} \label{es}



Let $C$ be the cusp of Example \ref{ex-cusp}, locally defined by the polynomial $f=x^3-y^2 \in \C[[x,y]]$. Its semigroup is $$\Gamma = \{ 0, 2, 3, 4 \dots \}$$ and $H = \{ 1 \}$.
We know that 
\begin{center}
$\Hilb_{0}^0(C) = \Hilb_{0,1}^0(C) = \{ (1) \}$\\
\vspace{1mm}
$[\Hilb_{0,1}^0(C)] = 1$\\
\vspace{3mm}
$\Hilb_{0}^1(C) =\Hilb_{0,2}^1(C) = \{ (t^2, t^3) \}$\\
\vspace{1mm}
$[\Hilb_{0,2}^1(C)] = 1$.    
\end{center}
We start with $k=1$.
For every $n \geq 2$, from Remark \ref{simplify-reps} we have that $ \{n + H \} \cap \Gamma = \{ n+1 \}$ and
\begin{center}
$\PHilb{n} = \{ (t^n + \lambda t^{n+1} ) \}$\\
\vspace{1mm}
$[\PHilb{n}] = \Li$.
\end{center}
Then, for $k=2$, the only $2$-uple contributing to length $n$ is $(n_1 = n+1 , n_2 = n+2)$ and from Corollary \ref{expression-monomial} we have that
\begin{center}
$\{n_1 + H \} \cap \Gamma_{\underline{n}} = \{ n_1 +1 \} \, , \; \{n_2 + H \} \cap \Gamma_{\underline{n}} = \{ n_2 +1 \}$\\
\vspace{1mm}
$[\Hilb_{0,2}^n(C)] = [\PHilb{n_1}][\PHilb{n_2}] \Li^{-2} = \Li\cdot \Li\cdot \Li^{-2} = 1$\\
\vspace{1mm}
$\Hilb_{0,2}^n(C) = \{ (t^{n+1} , t^{n+2} ) \}$.
\end{center}
Substituting $\omega(\Li)=t^2$ the generating series of weight polynomials is
$$ (\star) = 1\cdot s^0 (1 + tv^2) 
+ 1 \cdot (1 + tv^2) (1 + t^3v^2) \left(\sum_{l \geq 1} s^{2l} \right) 
+ t^2 \cdot (1 + tv^2) \left(\sum_{l \geq 2} s^{2l} \right) 
$$
$$ = \frac{1+ tv^2}{1-s^2} (1 + t^{3}v^{2}s^{2} + t^{2}s^{4}).$$
The Milnor number of $C$ is $\mu = (2-1)(3-1) = 2$. Therefore, multiplying by $\frac{v}{s}$ we obtain 
$$\mathcal{P}_C = \frac{v^{-1}+ tv}{s^{-1}-s} (v^2s^{-2} + t^2v^2s^2 + v^4t^3).$$
which matches the reduced superpolynomial
of \cite[Section 6.4, Proposition]{super} under the change of variables $a=v$, $q=s$ and the normalization $\mathcal{P}_{\text{unknot}} = 1$.
Finally, setting $t=-1$ we obtain $$P_C = \frac{v^{-1}-v}{s^{-1}-s} (v^{2}s^{-2} +v^{2}s^{2} -v^{4})$$
which matches the expression for HOMFLY-PT polynomial of \cite[Equation (13)]{miglio}.
\end{ex}

\begin{ex}
Let $C$ be the $(3,4)$-curve, locally defined by the polynomial $f=x^3-y^4 \in \C[[x,y]]$. 
Its semigroup is $$\Gamma = \{ 0, 3, 4, 6, 7, 8, \dots \}$$ and $H = \{ 1, 2, 5\}$.
We know that 
\begin{center}
$\Hilb_{0}^0(C) = \Hilb_{0,1}^0(C) = \{ (1) \}$\\
\vspace{1mm}
$[\Hilb_{0,1}^0(C)] = 1$\\
\vspace{3mm}
$\Hilb_{0}^1(C) = \Hilb_{0,2}^1(C) = \{ (t^3, t^4) \}$\\
\vspace{1mm}
$[\Hilb_{0,2}^1(C)] = 1$.    
\end{center}
We start with $k = 1$. From Remark \ref{simplify-reps} we have that
\begin{center}
$\PHilb{3} = \{ (t^3 + \lambda_1 t^{4} + \lambda_5 t^8 ) \}$\\
\vspace{1mm}
$[\PHilb{2}] = \Li^2$\\
\vspace{3mm}
$\PHilb{4} = \{ (t^4 + \lambda_2 t^{6} + \lambda_5 t^9 ) \}$\\
\vspace{1mm}
$[\PHilb{n}] = \Li^2$\\
\vspace{3mm}
$\PHilb{n} = \{ (t^n + \lambda_1 t^{n+1} + \lambda_2 t^{n+2} + \lambda_5 t^{n+5} ) \}$ for $n \geq 6$\\
\vspace{1mm}
$[\PHilb{n}] = \Li^3$ for $n \geq 6$.
\end{center}

\begin{itemize}
    \item For $n = 2$, we are left to check $k=2$ and there are no $3$-uples contributing to length $n=2$. The only $2$-uple contributing are $(n_1 = 4 , n_2 = 6)$ and $(n_1 = 3 , n_2 = 8)$ contributing. From Corollary \ref{expression-monomial} we have that
\begin{center}
$\{4 + H \} \cap \Gamma_{\underline{n}} = \{ 6, 9 \} \, , \; \{6 + H \} \cap \Gamma_{\underline{n}} = \{ 7, 8 , 11 \}$\\
\vspace{1mm}
$[\Hilb_{0,2}^{2, \underline{n}}(C)] = [\PHilb{n_1}][\PHilb{n_2}] \Li^{-5} = \Li^2\cdot \Li^3\cdot \Li^{-5} = 1$\\
\vspace{1mm}
$\Hilb_{0,2}^{2, \underline{n}}(C) = \{ (t^{4} , t^{6} ) \}$.
\end{center}

\begin{center}
$\{3 + H \} \cap \Gamma_{\underline{n}} = \{ 8 \} \, , \; \{8 + H \} \cap \Gamma_{\underline{n}} = \{ 9, 10, 13 \}$\\
\vspace{1mm}
$[\Hilb_{0,2}^{2, \underline{n}}(C)] = [\PHilb{n_1}][\PHilb{n_2}] \Li^{-4} = \Li^2\cdot \Li^3\cdot \Li^{-4} = \Li$\\
\vspace{1mm}
$\Hilb_{0,2}^{2, \underline{n}}(C) = \{ (t^{3} +\lambda_1 t^4 , t^{8} ) \}$.
\end{center}
In total, $[\Hilb_{0}^{2}(C)] = [\Hilb_{0,2}^{2}(C)] = \Li +1$.

\item For $n = 3$, we are left to check $k=2$ and $k=3$. We start with $k=2$.
The only $2$-uple contributing is $(n_1 = 4 , n_2 = 9)$ and from Corollary \ref{expression-monomial} we have that
\begin{center}
$\{4 + H \} \cap \Gamma_{\underline{n}} = \{ 9 \} \, , \; \{9 + H \} \cap \Gamma_{\underline{n}} = \{ 10, 11, 14 \}$\\
\vspace{1mm}
$[\Hilb_{0,2}^{3}(C)] = [\PHilb{n_1}][\PHilb{n_2}] \Li^{-4} = \Li^2\cdot \Li^3\cdot \Li^{-4} = \Li$\\
\vspace{1mm}
$\Hilb_{0,2}^{3}(C) = \{ (t^{4} + \lambda_2 t^6 , t^{9} ) \}$.
\end{center}
Then, for $k=3$, the only $3$-uple contributing is $(n_1 = 6 , n_2 = 7, n_3 = 8)$ and from Corollary \ref{expression-monomial} we have that
\begin{center}
$\{6 + H \} \cap \Gamma_{\underline{n}} = \{ 7, 8, 11\} \, , \; \{7 + H \} \cap \Gamma_{\underline{n}} = \{ 8, 9, 12 \} \, , \; \{8 + H \} \cap \Gamma_{\underline{n}} = \{ 9, 10, 13 \}$\\
\vspace{1mm}
$[\Hilb_{0,3}^{3}(C)] = [\PHilb{n_1}][\PHilb{n_2}] [\PHilb{n_3}] \Li^{-9} = \Li^3\cdot \Li^3\cdot \Li^3\cdot \Li^{-9} = 1$\\
\vspace{1mm}
$\Hilb_{0,3}^{3}(C) = \{ (t^6, t^7, t^8 ) \}$.
\end{center}
In total, $[\Hilb_{0}^{3}(C)] = [\Hilb_{0,1}^{3}(C)] + [\Hilb_{0,2}^{3}(C)] + [\Hilb_{0,3}^{3}(C)] = \Li^2 + \Li +1$.

\item For $n = 4$, we are left to check $k=2$ and $k=3$. We start with $k=2$. The only $2$-uples contributing are $(n_1 = 6 , n_2 = 7)$ and $(n_1 = 6 , n_2 = 8)$.
From Corollary \ref{expression-monomial} we have that
\begin{center}
$\{6 + H \} \cap \Gamma_{\underline{n}} = \{ 7, 11 \} \, , \; \{7 + H \} \cap \Gamma_{\underline{n}} = \{ 9, 12 \}$\\
\vspace{1mm}
$[\Hilb_{0,2}^{4, \underline{n}}(C)] = [\PHilb{n_1}][\PHilb{n_2}] \Li^{-4} = \Li^3\cdot \Li^3\cdot \Li^{-4} = \Li^2$\\
\vspace{1mm}
$\Hilb_{0,2}^{4, \underline{n}}(C) = \{ (t^{6} + \lambda_2^6 t^8 , t^{7} + \lambda_2^7 t^8 ) \}$
\end{center}
and 
\begin{center}
$\{6 + H \} \cap \Gamma_{\underline{n}} = \{ 8, 11 \} \, , \; \{8 + H \} \cap \Gamma_{\underline{n}} = \{ 9, 10, 13 \}$\\
\vspace{1mm}
$[\Hilb_{0,2}^{4, \underline{n}}(C)] = [\PHilb{n_1}][\PHilb{n_2}] \Li^{-5} = \Li^3\cdot \Li^3\cdot \Li^{-5} = \Li$\\
\vspace{1mm}
$\Hilb_{0,2}^{4, \underline{n}}(C) = \{ (t^{6} + \lambda_1 t^7 , t^{8} ) \}$.
\end{center}
In total, $[\Hilb_{0,2}^4(C)] = \Li^2 + \Li$.
Then, for $k=3$, the only $3$-uple contributing is $(n_1 = 7 , n_2 = 8, n_3 = 9)$ and from Corollary \ref{expression-monomial} we have that 
\begin{center}
$\{7 + H \} \cap \Gamma_{\underline{n}} = \{ 8, 9, 12\} \, , \; \{8 + H \} \cap \Gamma_{\underline{n}} = \{ 9, 10, 13 \} \, , \; \{9 + H \} \cap \Gamma_{\underline{n}} = \{ 10, 11, 14 \}$\\
\vspace{1mm}
$[\Hilb_{0,3}^{4}(C)] = [\PHilb{n_1}][\PHilb{n_2}] [\PHilb{n_3}] \Li^{-9} = \Li^3\cdot \Li^3\cdot \Li^3\cdot \Li^{-9} = 1$\\
\vspace{1mm}
$\Hilb_{0,3}^{4}(C) = \{ (t^7, t^8, t^9 ) \}$.
\end{center}
In total, $[\Hilb_{0}^{4}(C)] = [\Hilb_{0,1}^{4}(C)] + [\Hilb_{0,2}^{4}(C)] + [\Hilb_{0,3}^{4}(C)] = 2\Li^2 + \Li +1$.

\item For $n = 5$, we are left to check $k=2$ and $k=3$. We start with $k=2$. The only $2$-uples contributing are $(n_1 = 6 , n_2 = 11)$, $(n_1 = 7 , n_2 = 8)$, and $(n_1 = 7 , n_2 = 9)$. From Corollary \ref{expression-monomial} we have that
\begin{center}
$\{6 + H \} \cap \Gamma_{\underline{n}} = \{ 11 \} \, , \; \{11 + H \} \cap \Gamma_{\underline{n}} = \{ 12, 13, 16 \}$\\
\vspace{1mm}
$[\Hilb_{0,2}^{5, \underline{n}}(C)] = [\PHilb{n_1}][\PHilb{n_2}] \Li^{-4} = \Li^3\cdot \Li^3\cdot \Li^{-4} = \Li^2$\\
\vspace{1mm}
$\Hilb_{0,2}^{5, \underline{n}}(C) = \{ (t^{6} + \lambda_1 t^7 + \lambda_2 t^8 , t^{11} ) \}$
\end{center}
and
\begin{center}
$\{7 + H \} \cap \Gamma_{\underline{n}} = \{ 9, 12 \} \, , \; \{9 + H \} \cap \Gamma_{\underline{n}} = \{ 10, 11, 14 \}$\\
\vspace{1mm}
$[\Hilb_{0,2}^{5, \underline{n}}(C)] = [\PHilb{n_1}][\PHilb{n_2}] \Li^{-5} = \Li^3\cdot \Li^3\cdot \Li^{-5} = \Li$\\
\vspace{1mm}
$\Hilb_{0,2}^{5, \underline{n}}(C) = \{ (t^{6} + \lambda_1 t^7 , t^{8} ) \}$
\end{center}
and
\begin{center}
$\{7 + H \} \cap \Gamma_{\underline{n}} = \{ 8, 12 \} \, , \; \{8 + H \} \cap \Gamma_{\underline{n}} = \{ 10, 13 \}$\\
\vspace{1mm}
$[\Hilb_{0,2}^{5, \underline{n}}(C)] = [\PHilb{n_1}][\PHilb{n_2}] \Li^{-4} = \Li^3\cdot \Li^3\cdot \Li^{-4} = \Li^2$\\
\vspace{1mm}
$\Hilb_{0,2}^{5, \underline{n}}(C) = \{ (t^{7} + \lambda_2^7 t^9 , t^{8} + \lambda_1^8 t^9 ) \}$.
\end{center}
In total, $[\Hilb_{0,2}^5(C)] = 2\Li^2 + \Li$.
Then, for $k=3$, the only $3$-uple contributing is $(n_1 = 8 , n_2 = 9, n_3 = 10)$ and from Corollary \ref{expression-monomial} we have that
\begin{center}
$\{8 + H \} \cap \Gamma_{\underline{n}} = \{ 9, 10, 13\} \, , \; \{9 + H \} \cap \Gamma_{\underline{n}} = \{ 10, 11, 14 \} \, , \; \{10 + H \} \cap \Gamma_{\underline{n}} = \{ 11, 12, 15 \}$\\
\vspace{1mm}
$[\Hilb_{0,3}^{5}(C)] = [\PHilb{n_1}][\PHilb{n_2}] [\PHilb{n_3}] \Li^{-9} = \Li^3\cdot \Li^3\cdot \Li^3\cdot \Li^{-9} = 1$\\
\vspace{1mm}
$\Hilb_{0,3}^{5}(C) = \{ (t^8, t^9, t^{10} ) \}$.
\end{center}
In total, $[\Hilb_{0}^{5}(C)] = [\Hilb_{0,2}^{5}(C)] + [\Hilb_{0,3}^{5}(C)] = 2\Li^2 + \Li +1$.

\item For $n \geq 6$, we are left to check $k=2$ and $k=3$. We start with $k=2$. The only $2$-uples contributing are $(n_1 = n+1 , n_2 = n+6)$, $(n_1 = n+2 , n_2 = n+3)$ and $(n_1 = n+2 , n_2 = n+4)$. 
From Corollary \ref{expression-monomial} we have that
\begin{center}
$\{n+1 + H \} \cap \Gamma_{\underline{n}} = \{ n+1 +5 \} \, , \; \{n+6 + H \} \cap \Gamma_{\underline{n}} = \{ n+6 +1, n+6 +2, n+6 +5 \}$\\
\vspace{1mm}
$[\Hilb_{0,2}^{n, \underline{n}}(C)] = [\PHilb{n_1}][\PHilb{n_2}] \Li^{-4} = \Li^3\cdot \Li^3\cdot \Li^{-4} = \Li^2$\\
\vspace{1mm}
$\Hilb_{0,2}^{n, \underline{n}}(C) = \{ (t^{n+1} + \lambda_1 t^{n+1+1} + \lambda_2 t^{n+1+2} , t^{n+6} ) \}$
\end{center}
and
\begin{center}
$\{n+2 + H \} \cap \Gamma_{\underline{n}} = \{ n+2 +2, n+2 +5 \} \, , \; \{n+4 + H \} \cap \Gamma_{\underline{n}} = \{ n+4+1 , n+4 +2 , n+4+5 \}$\\
\vspace{1mm}
$[\Hilb_{0,2}^{n, \underline{n}}(C)] = [\PHilb{n_1}][\PHilb{n_2}] \Li^{-5} = \Li^3\cdot \Li^3\cdot \Li^{-5} = \Li$\\
\vspace{1mm}
$\Hilb_{0,2}^{n, \underline{n}}(C) = \{ (t^{n+2} + \lambda_1 t^{n+2+1} , t^{n+4} ) \}$
\end{center}
and
\begin{center}
$\{n+2 + H \} \cap \Gamma_{\underline{n}} = \{ n+2+1, n+2+5 \} \, , \; \{n+3 + H \} \cap \Gamma_{\underline{n}} = \{ n+3 +2, n+3+5 \}$\\
\vspace{1mm}
$[\Hilb_{0,2}^{n, \underline{n}}(C)] = [\PHilb{n_1}][\PHilb{n_2}] \Li^{-4} = \Li^3\cdot \Li^3\cdot \Li^{-4} = \Li^2$\\
\vspace{1mm}
$\Hilb_{0,2}^{n, \underline{n}}(C) = \{ (t^{n+2} + \lambda_2^{n+2} t^{n+2+2} , t^{n+3} + \lambda_1^{n+3} t^{n+3+1} ) \}$.
\end{center}
In total, $[\Hilb_{0,2}^n(C)] = 2\Li^2 + \Li$.
Then, for $k=3$, the only $3$-uple contributing is $(n_1 = n+3 , n_2 = n+4, n_3 = n+5)$ and from Corollary \ref{expression-monomial} we have that
\begin{center}
$\{n+3 + H \} \cap \Gamma_{\underline{n}} = \{ n+3 + 1, n+3 +2,  n+3 +5\} \, , \; \{n+4 + H \} \cap \Gamma_{\underline{n}} = \{ n+4 + 1, n+4 +2,  n+4 +5 \} \, , \; \{n+5 + H \} \cap \Gamma_{\underline{n}} = \{ n+5 + 1, n+5 +2,  n+5 +5 \}$\\
\vspace{1mm}
$[\Hilb_{0,3}^{n}(C)] = [\PHilb{n_1}][\PHilb{n_2}] [\PHilb{n_3}] \Li^{-9} = \Li^3\cdot \Li^3\cdot \Li^3\cdot \Li^{-9} = 1$\\
\vspace{1mm}
$\Hilb_{0,3}^{n}(C) = \{ (t^{n+3}, t^{n+4}, t^{n+5} ) \}$.
\end{center}
In total, $[\Hilb_{0}^{n}(C)] = [\Hilb_{0,1}^{n}(C)] + [\Hilb_{0,2}^{n}(C)] + [\Hilb_{0,3}^{n}(C)] = \Li^3 + 2\Li^2 + \Li +1$.
\end{itemize}
Their generating series is
 $$ (\star) = 1\cdot s^0 (1 + tv^2) 
 $$
 $$+ 1\cdot s^2  (1 + tv^2) (1 + t^3v^2) 
 $$
 $$+ (t^2 + 1)\cdot s^4 (1 + tv^2) (1 + t^3v^2) 
 $$
 $$ + s^6 \bigl( t^4 \cdot (1 + tv^2) + t^2 \cdot (1 + tv^2) (1 + t^3v^2) + 1\cdot (1 + tv^2) (1 + t^3v^2) (1 + t^5v^2)  \bigr) 
 $$
 $$ \hspace{-2mm} + s^8 \bigl( t^4 \cdot (1 + tv^2) + (t^4 + t^2)\cdot (1 + tv^2) (1 + t^3v^2) + 1 \cdot (1 + tv^2) (1 + t^3v^2) (1 + t^5v^2)  \bigr) 
  $$
  $$ + s^{10} \bigl( (2t^4 + t^4) \cdot (1 + tv^2) (1 + t^3v^2) + 1 \cdot (1 + tv^2) (1 + t^3v^2) (1 + t^5v^2)  \bigr) 
$$
$$ + t^6 \cdot (1 + tv^2) \left(\sum_{l \geq 6} s^{2l} \right) 
+ (2t^4 + t^2) \cdot (1 + tv^2) (1 + t^3v^2) \left(\sum_{l \geq 6} s^{2l} \right) 
$$
$$ + 1\cdot (1 + tv^2) (1 + t^3v^2)(1 + t^5v^2) \left(\sum_{l \geq 6} s^{2l} \right) 
$$
$$=\frac{1 + tv^2}{1 - s^2} \Bigg(
1 + s^{2} t^{3} v^{2} + s^{4} \left(t^{2} + t^{5} v^{2}\right) + s^{6} \left(t^{4} + t^{5} v^{2} + t^{8} v^{4}\right) + s^{8} \left(t^{4} + t^{7} v^{2}\right)$$
$$+ s^{10} t^{7} v^{2} + s^{12} t^{6}
\Bigg) .
$$
The Milnor number of $C$ is $\mu =(3-1)(4-1)=6$. Therefore, multiplying by $(\frac{v}{s})^5$ we obtain
$$\mathcal{P}_C = 
\frac{v+ tv^{-1}}{s-s^{-1}}
\Biggl(
s^{-6}v^{6} 
 + s^{-4} t^{3} v^{8} + s^{-2} \left(t^{2}v^6 + t^{5} v^{8}\right) + t^{4}v^6 + t^{5} v^{8} + t^{8} v^{10} $$
 $$+ s^{2} \left(t^{4} v^6 + t^{7} v^{8}\right) + s^{4} t^{7} v^{8} + s^{6} t^{6} v^6
\Biggr)
$$
which matches the reduced superpolynomial $$\mathcal{P}_C = 
a^{6} q^{-6} + 
a^{6} q^{-2} t^{2} + 
a^{6} t^{4} + 
a^{6} q^{2} t^{4} + 
a^{6} q^{6} t^{6} + 
a^{8} q^{-4} t^{3} + 
a^{8} q^{-2} t^{5} + 
a^{8} t^{5} + $$
$$
a^{8} q^{2} t^{7} + 
a^{8} q^{4} t^{7} + 
a^{10} t^{8}$$
of \cite[Section 3.6, Example]{super} under the change of variables $a=v$, $q=s$ and the normalization $\mathcal{P}_{\text{unknot}} = 1$. Finally, setting $t=-1$ we obtain 
$$
P_C = 
\frac{v-v^{-1}}{s-s^{-1}}
\bigg(
v^{6} s^{-6} - v^{8} s^{-4} + \left(v^{6} - v^{8}\right) s^{-2} + \left(v^{6} - v^{8} + v^{10}\right) $$
$$+ \left(v^{6} - v^{8}\right) s^{2}- v^{8} s^{4} + v^{6} s^{6}
\bigg)  
$$
which matches the HOMFLY-PT polynomial $$P_C = a^{-6} z^{6} + \left(6 a^{-6} - a^{-8}\right) z^{4} + \left(10 a^{-6} - 5 a^{-8}\right) z^{2} + 5 a^{-6} - 5 a^{-8} + a^{-10}$$
of \cite{poly34}
under the change of variables $z = s - s^{-1}$, $a = v^{-1}$ and the normalization ${P}_{\text{unknot}} = 1$. 
\end{ex}

We now generalize Example \ref{es}.
\begin{ex}
Let $C$ be the $(2,2k+1)$-curve, locally defined by the polynomial $f=x^2-y^{2k+1} \in \C[[x,y]]$. Its semigroup is $$\Gamma = \{ 0, 2,  4 , 6, 8, \dots, 2k-2, 2k, 2k+1 , 2k+2 \dots \}$$ and $H = \{ 1 , 3, 5, 7, \dots , 2k-1\}$.
We know that 
\begin{center}
$\Hilb_{0}^0(C) = \Hilb_{0,1}^0(C) = \{ (1) \}$\\
\vspace{1mm}
$[\Hilb_{0,1}^0(C)] = 1$.
\end{center}
Then, we continue with $l=1$, where $l$ denotes the number of minimal generators, and $n \geq 2$.
\begin{itemize}
\item For $n=2j$, $1 \leq j \leq k-1$, from Remark \ref{simplify-reps} we have that $ \{n + H \} \cap \Gamma = \{ n + (2k - (n-1)), \dots , n + 2k - 1 \}$ and
\begin{center}
$[\PHilb{n}] = \Li^j$.
\end{center}

\item For $n \geq 2k$ from Remark \ref{simplify-reps} we have that $ \{n + H \} \cap \Gamma = \{ n + 1, \dots , n + 2k - 1 \}$ and
\begin{center}
$[\PHilb{n}] = \Li^k$.
\end{center}
\end{itemize}
Then, we move on to $l=2$.
\begin{itemize}
\item For $n=2j$, $1 \leq j \leq k-1$, the $2$-uples contributing are
$(n_1^i = n+2i , n_2^i = 2k+(n - (2i -1) ) )$ for $i = 1, \dots , j$. From Corollary \ref{expression-monomial} we have that
\begin{center}
$[\Hilb_{0,2}^{n}(C)] = \sum_{i=0}^{j-1} \Li^i$.
\end{center}

\item For $n=2j-1$, $1 \leq j \leq k$, the $2$-uples contributing are
$(n_1^i = n+(2i-1) , n_2^i = 2k+ (n - 2(i-1) ))$ for $i = 1, \dots , j$. From Corollary \ref{expression-monomial} we have that
\begin{center}
$[\Hilb_{0,2}^{n}(C)] = \sum_{i=0}^{j-1} \Li^i$.
\end{center}

\item For $n \geq 2k$,  the $2$-uples contributing are
$(n_1^i = n + i , n_2^i = n + 2k - (i-1) ))$ for $i = 1, \dots , k$. From Corollary \ref{expression-monomial} we have that
\begin{center}
$[\Hilb_{0,2}^{n}(C)] = \sum_{i=0}^{k-1} \Li^i$.
\end{center}
\end{itemize}
Their generating series is
 $$ (\star) = 1 \cdot s^0 (1 + tv^2) 
 $$
 $$ 
 + 
 \sum_{j=1}^{k-1}  s^{2(2j)} \bigg(
 t^{2j} \cdot (1 + tv^2) + 
\left(\sum_{i=0}^{j-1} t^{2i} \right)\cdot (1 + tv^2) (1 + t^3v^2) 
 \bigg)
 $$
$$ 
 + 
 \sum_{j=1}^{k}  s^{2(2j - 1)} 
\left(\sum_{i=0}^{j-1} t^{2i} \right)\cdot (1 + tv^2) (1 + t^3v^2) 
 $$
 $$ 
 + 
 \sum_{n \geq 2k}  s^{2n} \bigg(
 t^{2k} \cdot (1 + tv^2) + 
\left(\sum_{i=0}^{k-1} t^{2i} \right)\cdot (1 + tv^2) (1 + t^3v^2) 
 \bigg)
 $$
$$
 =\frac{1 + tv^2}{1 - s^2} \Bigg(
 \sum_{m=0}^{k}t^{2m}s^{4m}+\sum_{m=0}^{k-1}t^{2m+3}v^{2}s^{4m+2}
\Bigg) .$$
The Milnor number of $C$ is $\mu = (2-1)(2k+1 -1) = 2k$. Therefore, multiplying by $(\frac{v}{s})^{2k-1}$ we obtain
$$\mathcal{P}_C =
\frac{v + tv^{-1}}{s - s^{-1}} 
\Bigg(
\sum_{j=0}^{k} t^{2j} v^{2k} s^{4j - 2k} + \sum_{j=0}^{k-1} t^{2j+3} v^{2k+2} s^{4j+2-2k}
\Bigg)$$
$$
=
\frac{v + tv^{-1}}{s - s^{-1}} 
\Bigg(
 v^{2k}  s^{-2k} \frac{(t^2s^4)^{k+1} - 1}{t^2 s^4 -1}
 + t^3 v^{2k+2} s^{2-2k} \frac{(t^2s^4)^{k} - 1}{t^2 s^4 -1}
\Bigg) 
$$
which matches the reduced superpolynomial $$\mathcal{P}_C = a^{2k} \sum_{i=0}^{k} q^{4i - 2k} t^{2i} + a^{2k+2} \sum_{i=1}^{k} q^{4i - 2k - 2} t^{2i+1}$$
of \cite[Section 6.4, Proposition]{super} under the change of variables $a=v$, $q=s$ and the normalization $\mathcal{P}_{\text{unknot}} = 1$.
Finally, setting $t=-1$ we obtain 
$$P_C =
\frac{v - v^{-1}}{s - s^{-1}} 
\Bigg(
 v^{2k}  s^{-2k} \frac{s^{4(k+1)} - 1}{s^4 -1}
 - v^{2k+2} s^{2-2k} \frac{s^{4k} - 1}{s^4 -1}
\Bigg)
$$
and for $k=1$ it matches the result of Example \ref{es}. This also matches the HOMFLY-PT polynomial 
$$P_C = t^{-2k} \frac{s^{2k+4} - s^{-2k}}{s^4 -1}
 - t^{-2k-2}  \frac{s^{2k+2} - s^{-2k+2}}{s^4 -1}
 $$
of \cite[Example 6.3]{poly2k1}
under the change of variables $q = s$, $t = v^{-1}$ and the normalization ${P}_{\text{unknot}} = 1$. 
\end{ex}

\section{References}
\printbibliography[heading=none]

\end{document}